\documentclass[12pt,reqno]{amsart}

\usepackage[final,
color]{showkeys}
\definecolor{refkey}{gray}{.85}
\definecolor{labelkey}{gray}{.85}

\usepackage{AKstyle}

\numberwithin{equation}{section}

\usepackage{caption}
\usepackage[labelfont=rm]{subcaption}

\usepackage[pagebackref=true, colorlinks]{hyperref}

\hypersetup{pdffitwindow=true,linkcolor=blue,citecolor=blue,urlcolor=blue,menucolor=blue}

\usepackage{comment}

\begin{document}

\author{Dubi Kelmer}
\thanks{Kelmer is partially supported by NSF CAREER grant  DMS-1651563.}
\email{kelmer@bc.edu}
\address{Boston College, Boston, MA}
\author{Alex Kontorovich}
\thanks{Kontorovich is partially supported by
an NSF CAREER grant DMS-1455705, an NSF FRG grant DMS-1463940,  a BSF grant, a Simons Fellowship,  a von Neumann Fellowship at IAS, and the IAS's NSF grant DMS-1638352.}
\email{alex.kontorovich@rutgers.edu}
\address{Rutgers University, New Brunswick, NJ and Institute for Advanced Study, Princeton, NJ}

\title[Exponents for the Equidistribution of Shears]
{Exponents for the Equidistribution of Shears and Applications}

\begin{abstract}
In \cite{KelmerKontorovich2017}, the authors introduced ``soft'' methods to prove the effective (i.e. with power savings error) equidistribution of ``shears'' in cusped hyperbolic surfaces.
In this paper, we study the same problem but
now
 allow full use of the spectral theory of automorphic forms 
 to produce 
explicit
 exponents, 
 and
 uniformity in 
 parameters. 
We 
give applications to counting 
square values of   quadratic forms.
\end{abstract}
\date{\today}
\maketitle
\tableofcontents


\section{Introduction}

\subsection{Equidistribution}
Let $\G<G:=\SL_2(\R)$ be a non-uniform lattice, and let 
$$
\bx_0\in T^1(\G\bk\bH)\cong\G\bk G
$$ 
be a base point in the unit tangent bundle of the punctured surface $\G\bk\bH$, so that the visual limit point 
$$
\fa=\lim_{t\to\infty}\bx_0\cdot a_t
$$
is a cusp of $\G$. Here $a_t=\diag(e^{t/2},e^{-t/2})$ is the geodesic flow; let 
$A^+=\{a_t,t>0\}$. 
As in \cite{KelmerKontorovich2017}, we define a {\bf shear} of the cuspidal geodesic ray $\bx_0\cdot A^+$
to be its left-translate by 
$$
\fs_T:=\mattwo{(T^2+1)^{-1/4}}{T(T^2+1)^{-1/4}}{0}{(T^2+1)^{1/4}}.
$$
For example, identifying $G/K$ with $\bH$ under $g\mapsto g \cdot i$ (here $K=\SO(2)$ is a maximal compact subgroup) and taking  $\bx_0=e$, the identity element of $G$, we have for a right-$K$-invariant test function $\Psi\in C_c(\G\bk G)^K$ that its evaluation along such a shear is given by
$$
\int_{a\in A^+}\Psi(\bx_0\cdot a\cdot \fs_T)da = 
\int_{1/\sqrt{T^2+1}
}^\infty
\Psi\left(Ty+i{y}\right){dy\over y}.
$$
In \cite{KelmerKontorovich2017}, the authors proved the effective equidistribution of  shears (as $T\to\infty$) using ``soft'' ergodic methods (e.g. mixing) and basic properties of Eisenstein series.\footnote{See also \cite{OhShah2014}, where an asymptotic formula is obtained by different means, with an  error term weaker than power savings.}
The goal of this paper is to use more of the  spectral theory of automorphic forms 
to produce explicit 
exponents in this problem.
For ease of exposition, and to write the best exponents that come from our method, we restrict below to $\G$ conjugate to the congruence group $\G_0(p)$.
A special case of \thmref{t:Equidistribution1} below gives the following  bound.
\begin{thm}\label{thm:1}
Assume the Ramanujan conjecture for the exponent bounding the Fourier coefficients of Maass forms on $\G$ (see \secref{sec:3}). 
Then for any $\Psi\in C^\infty_c(\G\bk G)^K$, any $T\ge2$, and any $\gep>0$, there are constants $C_j=C_j(\Psi)$, $j=1,2$, so that
$$
\int_{a\in A^+}\Psi(\bx_0\cdot a\cdot \fs_T)da = 
C_1\log T + C_2 + O_{\Psi,\gep}(T^{-1/4+\gep}).
$$
\end{thm}

\begin{rmk}
We take the opportunity here to correct an error in the analysis in \cite{KelmerKontorovich2017}, which has no effect on the qualitative power gain, but does affect the exponents as explicitly quantified here. In particular, \cite[Remark 1.7]{KelmerKontorovich2017} is incorrect as stated, and we do not know how to obtain square-root cancellation by this approach. See \rmkref{rmk:error} for the error and how to correct it.
\end{rmk}

\subsection{Counting}
As is standard, such equidistribution results can be applied to counting problems in discrete orbits. In particular, see \cite[\S1.3.1]{KelmerKontorovich2017} where we explain that proving the effictive equidistribution of shears settles the remaining lacunary cases of the Duke-Rudnick-Sarnak/Eskin-McMullen program \cite{DukeRudnickSarnak1993, EskinMcMullen1993, Margulis2004} of effectively counting discrete orbits on quadrics in archimedean balls. 
In smooth form, one can produce from \thmref{thm:1} above some 
rather sharp error exponents, as we now illustrate.

Let $F$ be a real ternary indefinite quadratic form, let $G=\SO_F^\circ(\R)$ be the connected component of the real special orthogonal group preserving $F$, and assume that $\G<G$ is the image of $\G_0(p)$ under a spin morphism $\PSL_2(\R)\to\SO_F^\circ(\R)$ (see \secref{sec:spin}).
Let $\psi:G\to\R_{\ge0}$ be a smooth bump function about a sufficiently small bi-$K$-invariant neighborhood of the identity in $G$ (that is, a region of the form \eqref{e:Bdelta}), with $\int_G\psi=1$. 
Fix $\bv_0\in\R^3$ so that the orbit 
$$
\cO:=\bv_0\cdot\G \ \subset \ \R^3
$$ 
is discrete, the stabilizer of $\bv_0$ in $G$ is a split torus, $H$, say, and $\G\cap H$ is finite. 
Fix a right-$K$-invariant archimedean norm $\|\cdot\|$ on $H\bk G\cong \bv_0\cdot G$. This data  induces a smoothed indicator function of a norm-$T$ ball on $\bv\in\bv_0\cdot G$ via convolution with $\psi$:
\be\label{eq:psiTilDef}
\widetilde\psi_T(\bv):=\int_{g\in G}\bo_{\{\|\bv g\|<T\}}\psi(g)dg.
\ee
In \secref{sec:5} (see \propref{prop:smoothC}), we prove the following
\begin{thm}\label{thm:2}
Again assume the Ramanujan conjecture for Maass forms on $\G$.
Then for any $\gep>0$, 
there are constants $C_j=C_j(\psi)$, $j=1,2$, so that 
$$
\sum_{\bv\in\cO}\widetilde\psi_T( \bv) = C_1 T\log T + C_2 T +O_{\gep,\psi}\left(T^{\tfrac{3}{4}+\gep}\right).
$$
Unconditionally, the error exponent $\frac{3}{4}+\gep$ can be replaced by $\frac{3+4\gt}{4+8\gt}+\gep$, where $\gt=7/64$ is the best currently known bound towards the Ramanujan Conjecture (which stipulates that $\gt=0$ holds).
\end{thm}

\subsection{Explicit Constants}\label{sec:1.3}

In certain settings of classical interest, one can go a step further and explicitly identify the constants  $C_j$ appearing in the main terms above. 
To 
showcase
this fact, we
count
 integer points on the inhomogeneous Pythagorean quadric:
$$
W_d\ :\ x^2+y^2-z^2 = d,
$$
when $d$ is a perfect square (which corresponds to $\G\cap H$ finite as above). 
After unsmoothing the count in \thmref{thm:2} to make the constants independent of the smoothing function $\psi$, we obtain the following.
\begin{thm}\label{thm:3}
Let $d=n^2$ be a square, with $n\in\Z_{>0}$. 
Define the counting function
$$
\cN_d(T) \ :=\ \# W_d(\Z)\cap\{x^2+y^2+z^2<T^2\}
.
$$
Again let $\gt=7/64$ be the  bound towards the Ramanujan Conjecture (that $\theta=0$).
For any $\eta<\frac{3}{40+72\theta}$, $\beta>\frac32+2\theta$, 
and $T\geq d^\gb$, 
 we have that
\be\label{eq:thm3}
\cN_{d}(T)
\ = \
\cM_d(T)+
O\left(T^{1-\eta}d^{\beta\eta}\right)
,
\ee
where the ``main term'' is given by
$$
\cM_d(T)\ := \ \frac{\sqrt{128}T}{\pi }\bigg(\log(T)+C
-D(n)+\log(2)(\tfrac{1}{3}-\tfrac{1}{2^{\nu+2}})
\bigg)
.
$$
Here
 $\nu=\nu_2(n)$ is the $2$-adic valuation of $n$,
the constant  $C$ is given by:
\begin{eqnarray*}
C
&=& 2\g-1
-2{\gz'\over\gz}(2)-\tfrac{\log(2)}{2}-\log\left(\tfrac{|\G(1/4)|^4}{4\pi^3}\right)
\ = \
0.616174\dots,
\end{eqnarray*}
where $\g=0.577\dots$ is Euler's constant, $\zeta(s)$ is the Riemann zeta function,  $\G(s)$ is the Gamma function, and $\phi$ is the Euler totient function,
and
the 
factor
 $D(n)$ is
 the Dirichlet coefficient of $(\gz^2\cdot\gz')(s-1)$, that is,
$$
D(n)=\frac 1n \sum_{a|n}\phi(\tfrac{n}{a})\log(a).
$$
\end{thm}

To illustrate the validity of this complicated formula, we verify it numerically with plots of $\cN_d(T)$, $\cM_d(T)$, and their difference, for $d=144$; see \figref{fig:1}. For $T$ as large as $10,000$, the counting function reaches around $350,000$, while the difference $\cN_d(T)-\cM_d(T)$ remains of size around $400\asymp \sqrt T$, suggesting perhaps that \eqref{eq:thm3} may remain valid with any $\eta<1/2$ and $\gb=1/2$.

\begin{rem}
Our results are meaningful as long as $T>d^{3/2+\theta+\epsilon}$. In particular, assuming Ramanujan we can take $T$ to be almost as small as $d^{3/2}$. When $T<\sqrt d$ we trivially have $\cN_{d}(T)=0$ and it is an interesting problem to obtain meaningful asymptotics also for the range $\sqrt{d}<T<d^{3/2}$. (In the rather different setting of $d$  being a fundamental discriminant, Friedlander-Iwaniec \cite{FriedlanderIwaniec2013}, using different tools showed an asymptotic formula which is effective for $T$ almost as small as $\sqrt d$.)
\end{rem}

\begin{figure}
        \begin{subfigure}[t]{0.45\textwidth}
                \centering
		\includegraphics[width=\textwidth]{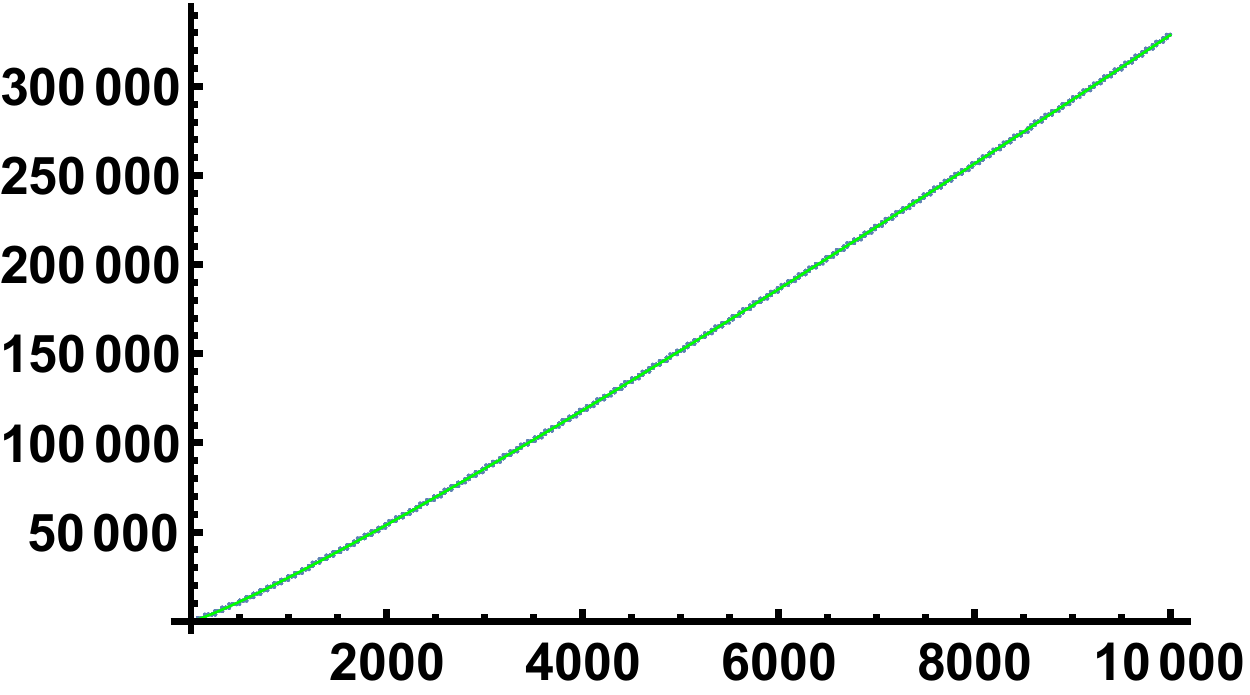}
                \caption{\includegraphics[width=.1in]{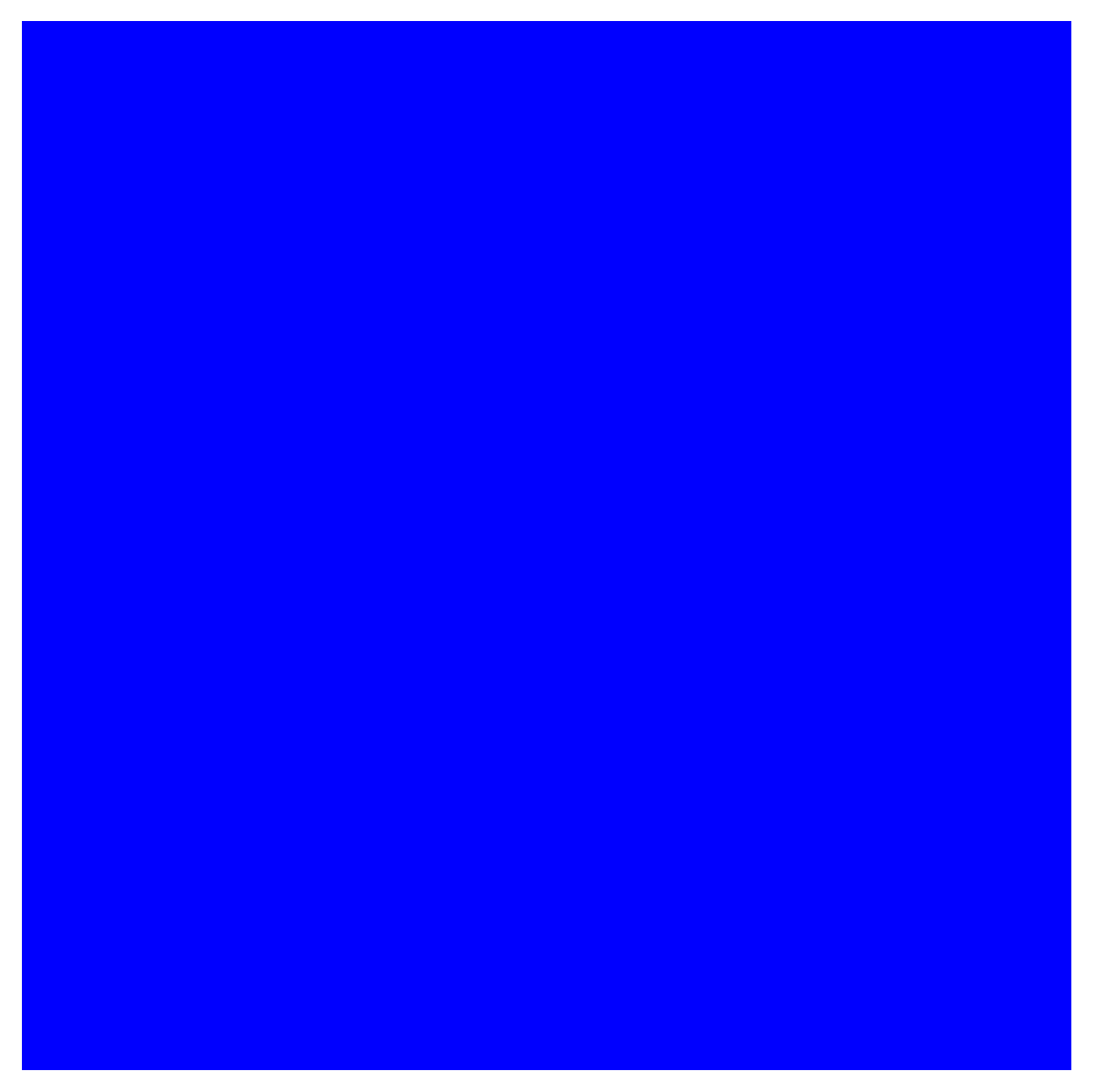}=$\cN_d(T)$ and  \includegraphics[width=.1in]{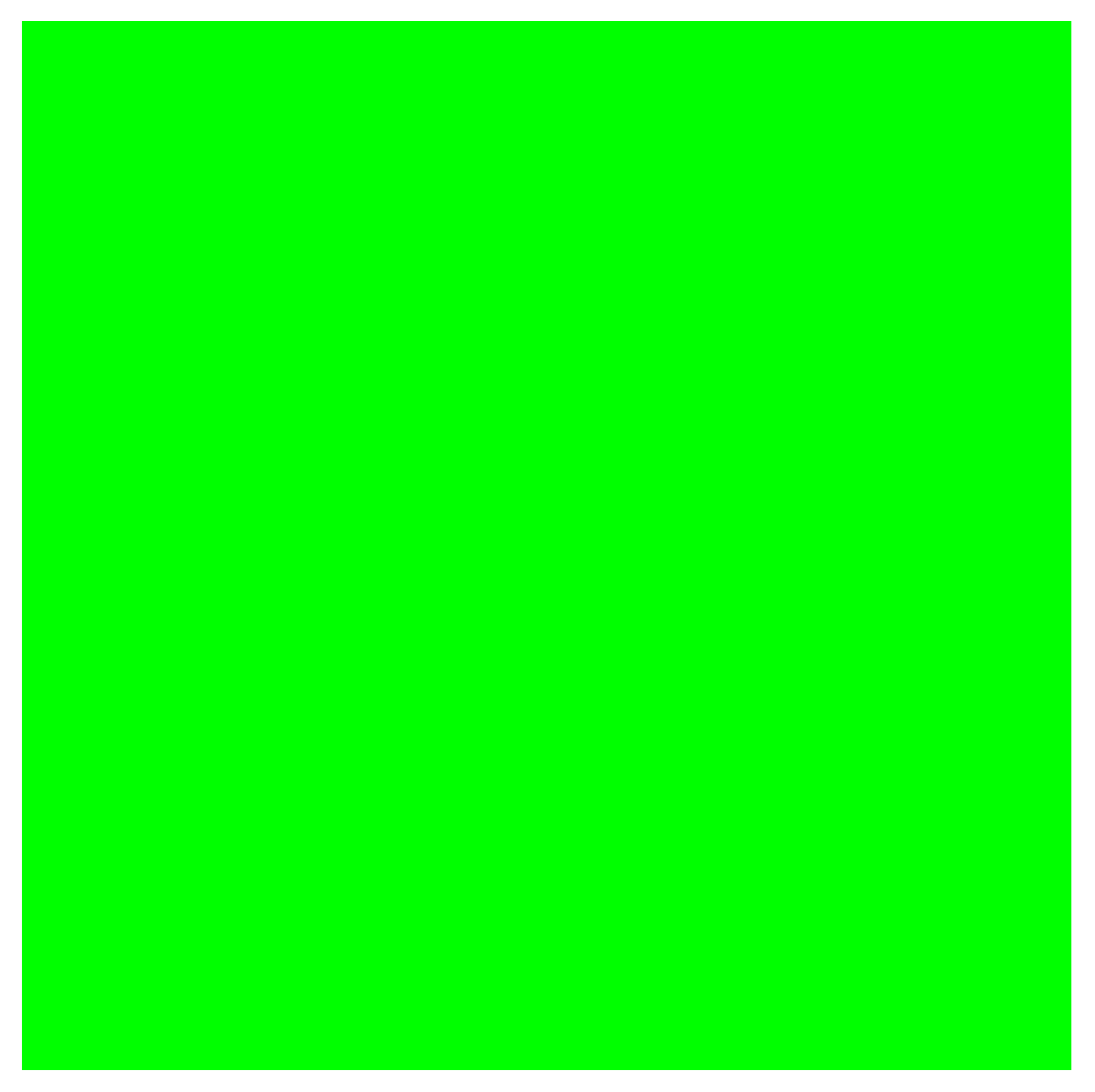}=$\cM_d(T)$ \linebreak \phantom{asdfasdf} (indistinguishable)}
        \end{subfigure}%
\qquad
        \begin{subfigure}[t]{0.45\textwidth}
                \centering
		\includegraphics[width=\textwidth]{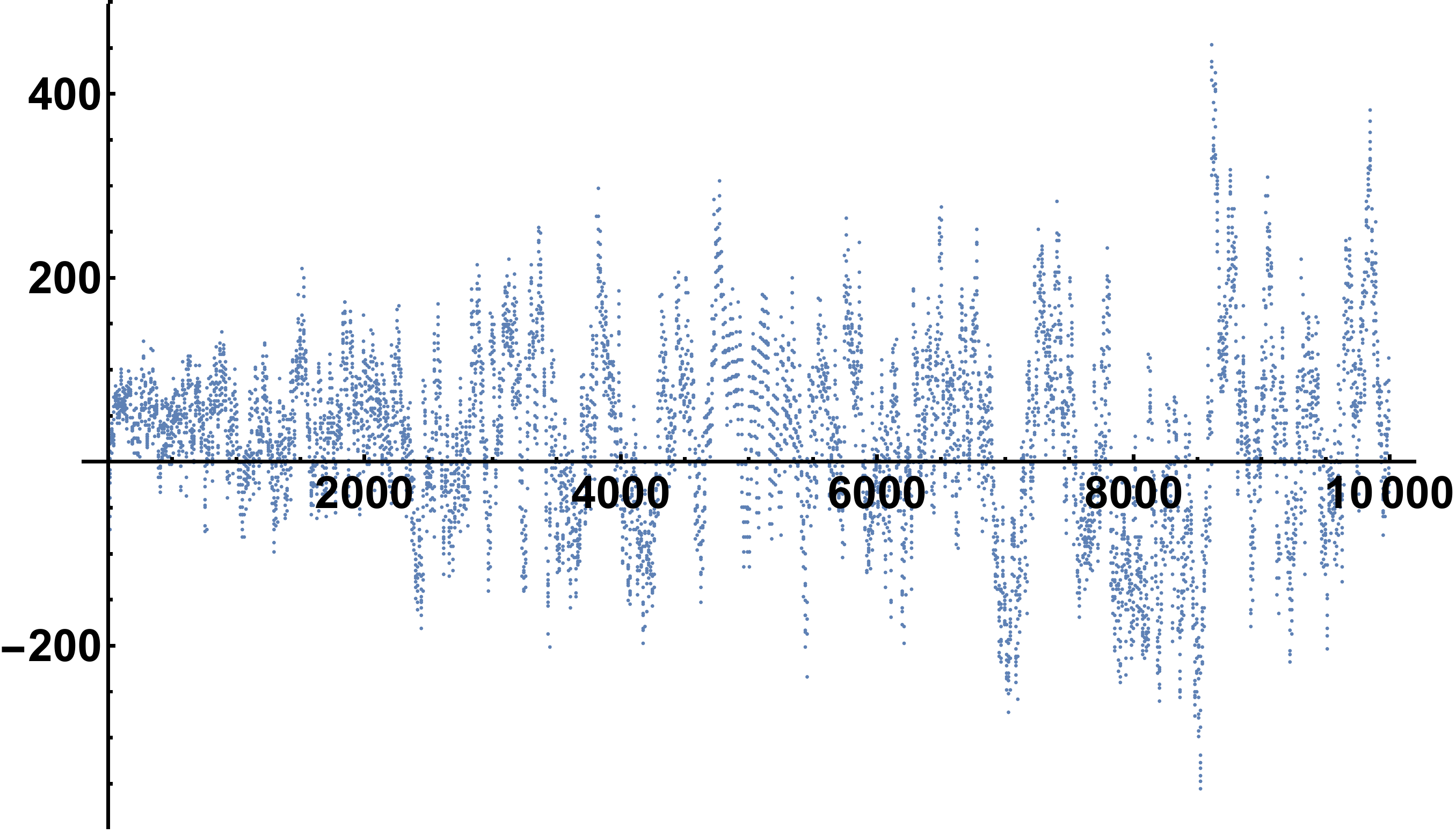}
                \caption{$\cN_d(T)-\cM_d(T)$}
        \end{subfigure}%

\caption{Plots for $d=12^2=144$ and $T<10,000$.}
\label{fig:1}
\end{figure}

 
 Along the way to proving \thmref{thm:3}, we need to establish the following result,  counting the number of binary quadratic forms of a fixed square discriminant $d$ with coefficients in a norm ball. As it is possibly of independent interest,  we record it here.
Let  
$$
Q(a,b,c)=b^2-4ac,
$$ 
and
\begin{equation}
\cN_{Q,d}(T)=\{(a,b,c)\in \Z^3:Q(a,b,c)=d,\; 2a^2+b^2+2c^2\leq T^2\}.
\end{equation}
\begin{thm}\label{thm:4}
With all notation and assumptions as in \thmref{thm:3}, we have
\be\label{e:NQ2}
\cN_{Q,d}(T)=\frac{\sqrt{72}T}{\pi}\left(\log(T)+C-D(n)+O(\frac{d^{\eta\beta}}{T^\eta})\right).
\ee
\end{thm}

As discussed in \cite[Remark 1.10]{KelmerKontorovich2017}, there are 
many other methods for counting such expressions. For example, Hulse et al \cite{HulseEtAl2016} used a Multiple Dirichlet Series technique to count binary forms with a fixed discriminant. In smooth form and counting in a slightly different region, they obtain a version of \eqref{e:NQ2} with square-root error in the $T$ aspect, but with no visible uniformity in $d$. It is also not clear how easy it would be to convert their constants into the completely numerically explicit values as in \eqref{e:NQ2}.

\subsection{Notation}
Throughout this paper we denote by $G=\PSL_2(\R)$ and $\G$ a non-uniform lattice in $G$. For ease of exposition (and for our applications), we assume that $\G$ is conjugate to $\G_0(p)$; minor modifications are needed to handle a more general setting. 
We will use the notation $A(t)\ll B(t)$ to mean that there is some constant $c>0$ such that $A(t)\leq cB(t)$, and we will use subscripts to indicate the dependance of the constant on parameters. For $B(t)\geq 0$ the notation $A(t)=O(B(t))$ means that $|A(t)|\ll B(t)$. 
We write $A(t)\asymp B(t)$ for $A(t)\ll B(t) \ll A(t)$.

\subsection{Organization}
The preliminary \secref{sec:prelim}
reviews spectral decompositions, Eisenstein series, and Sobolev norms.
Then in \secref{sec:3}, we improve on our analysis in \cite{KelmerKontorovich2017}, giving stronger estimates for Fourier expansions at various cusps in terms of approximations to the Ramanujan Conjectures. We use these, together with some slight improvements on the method in \cite{KelmerKontorovich2017}, to prove  in \secref{sec:edShears} the sharp equidistribution of shears claimed in \thmref{thm:1}. The counting \thmref{thm:2} is derived from this in \secref{sec:5}, and then used in \secref{sec:6} to prove the explicit counting \thmsref{thm:3} and \ref{thm:4}.
Calculations of Fourier expansions for Eisenstein series on $\G_0(p)$ are reserved for the Appendix.
%

\subsection*{Acknowledgements}

We thank Zeev Rudnick for comments on an earlier draft.


\section{Preliminaries}\label{sec:prelim}
\subsection{Coordinates}
Let $K,A,N\leq G$ denote the orthogonal group, the diagonal group, and the unipotent group respectively.
Explicitly let 
$$k_\theta=\left(\begin{smallmatrix} \cos\theta& \sin\theta \\-\sin\theta& \cos\theta\end{smallmatrix}\right),\; a_y=\left(\begin{smallmatrix} y^{1/2} &0 \\0& y^{-1/2} \end{smallmatrix}\right),\mbox{ and } n_x=\left(\begin{smallmatrix} 1& x \\0& 1\end{smallmatrix}\right),$$ parametrize elements in $K,A$ and $N$ respectively.
The decomposition $G=NAK$ gives coordinates $g=n_xa_yk_\theta$ on $G$, and the Haar measure in these coordinates is
\begin{equation}\label{e:Haar}
dg=\frac{dxdyd\theta}{2\pi y^2}.
\end{equation}
The group $G$ acts on the upper half space  $\bH=\{z\in \C: \Im(z)>0\}$ by linear fractional transformations, explicitly, for $g=\left(\begin{smallmatrix} a& b \\c& d\end{smallmatrix}\right)$, $gz=\frac{az+b}{cz+d}$ preserving the hyperbolic area $d\mu=\frac{dxdy}{y^2}$.

For any lattice $\G$ we identify the quotient $\G\bk \bH$ with $\G\bk G/K$ (in particular, we will think of functions on $\G\bk \bH$ as right $K$-invariant functions on $\G\bk G$). 
The hyperbolic area $v_\G=\mu(\G\bk \bH)$ is the same as the Haar measure of $\G\bk G$.
\begin{rem}
For $\G_1=\SL_2(\Z)$ we recall that $v_{\G_1}=\frac\pi 3$.
\end{rem}
\subsection{Sobolev Norms}
Fix a basis $\sB=\{X_{1},X_{2},X_{3}\}$ for the Lie algebra $\mathfrak{g}$ of $G$, and given a smooth test function $\Psi\in C^{\infty}(\G\bk G)$, define the ``$L^{p}$, order-$d$'' Sobolev norm $\cS_{p,d}(\Psi)$  as
$$
\cS_{p,d}(\Psi) \ : = \ \sum_{\ord(\sD)\le d}\|\sD\Psi\|_{L^{p}(\G\bk G)}
.
$$
Here $\sD$ ranges over monomials in $\sB$ of order at most $d$. Note that since the right action of $\sB$ commutes with the left action of $G$, all norms are invariant in the sense that $\cS_{p,d}(\Psi^\tau)=\cS_{p,d}(\Psi)$ where $\Psi^\tau(x)=\Psi(\tau x)$.

We will work with various norms that are convex combinations of these Soboev norms and it will be convinient to classify these norms with respect to how large they 
become
on functions approximating a small bump function. For small $\delta>0$ let 
\begin{equation}\label{e:Bdelta}
B_\delta=KA_\delta K\end{equation}
denote a (spherical) $\delta$-neighborhood of the identity,
 where 
 $$
 A_\delta=\{a_y: |\log(y)|<\delta\} .
 $$ 
 \begin{Def}
 We say that a norm $\cS$ is of {\bf degree} $\alpha$ if one can construct a family of smooth functions $\psi_\delta$ on $K\bk G/K$, supported on $B_\delta$, with average 
 $$
 \int_G \psi_\delta(g)dg=1,
 $$ 
 so that the corresponding periodized functions 
 $$\Psi_\delta(g)=\sum_{\g\in \G}\psi_\gd(\g g),$$ 
 have norms growing like $\cS(\Psi_\delta)\asymp \delta^{-\alpha}$. 
 We will slightly abuse notation and sometimes denote by $\cS_\alpha$ a norm of degree $\alpha$, without specifying 
the norm explicitly.
 \end{Def}
Note for future reference that the  
Sobolev norm $\cS_{p,d}$ is of degree $d+2-\frac{2}{p}$, and that if $\cS_\alpha$ and $\cS_\beta$ are of degrees $\alpha$ and $\beta$ respectively, then the convex combination $\cS_\alpha^q\cS_\beta^{1-q}$ is of degree $q\alpha+(1-q)\beta$. 
In particular, the $L^2$ norm has degree $1$, while the $L^\infty$ norm has degree $2$.

\subsection{Eisenstein series}
For any cusp $\fa$ of $\G$ let $\G_\fa$ denote the stabilizer of $\fa$ in $\G$ and $\tau_\fa\in G$ be a corresponding scaling matrix such that $\tau_\fa \infty=\fa$ and
$$\tau_\fa^{-1}\G\tau_\fa\cap N=\{n_k:k\in \Z\}.$$
In particular, for the congruence groups $\G_0(p)$ there are two cusps, one at $\infty$ of 
width 
$1$, and the other at $0$ of width $p$ with scaling matrix $\tau_0=\left(\begin{smallmatrix} 0 & 1/\sqrt{p} \\ -\sqrt{p} &0\end{smallmatrix}\right)=k_{\pi/2}\cdot a_{p}
$.

The Eisenstein series corresponding to a cusp $\fa$ is defined for $\Re(s)>1$ by
\begin{equation}\label{e:Eisenstein}
E_{\fa}(z,s)=\sum_{\g\in \G_\fa\bk \G}(\Im(\tau_\fa^{-1}\g z))^s
\end{equation}
and has a meromorphic continuation to $\C$ with a simple pole at $s=1$ with residue $\frac{1}{v_\G}$ and (since $\G$ is congruence) no other poles in $\Re(s)>\foh$. 
%

One can regularize the Eisenstein series by subtracting the pole at $s=1$, and we define the corresponding Kronecker limit by
\begin{equation}\label{e:tildeE}
\cK_{\G,\fa}(z)=\lim_{s\to 1}\left(E_{\G,\fa}(s,z)-\frac{1}{v_\G(s-1)}\right).
\end{equation}

\subsection{Spectral decomposition}
The hyperbolic Laplace operator $\gD$ 
acts (after unique extension) on the space
  $L^{2}(\G\bk \bH)$ of square-integrable automorphic functions, and is self-adjoint and positive semi-definite.
The spectrum of $\gD$ is composed of the constant functions, the continuous part (spanned by Eisenstein series), and discrete part (spanned, in the congruence case, by 
Maass cusp forms). 


We denote by $\cE(\G\bk \bH)$ the space spanned by the Eisenstein series
and by $\cC(\G\bk\bH)$ its orthogonal complement which is the space of cusp forms.
For $\G=\G_0(p)$ the space of cusp forms further decomposes into the space of old forms $\cC_{\rm old}(\G\bk \bH)$ spanned by the set
 $\{\vf(z),\vf(pz): \vf\in \cC(\G_1\bk \bH)\}$, and its orthogonal complement $\cC_{\rm new}(\G\bk \bH)$. 
 
For congruence $\G$, the space of cusp forms has an orthonormal basis composed of Hecke-Maass forms, that are joint eigenfunctions of the Laplacian and all Hecke operators.
We have the following spectral decomposition (see \cite[Theorems 4.7 and  7.3]{Iwaniec1995}).
\begin{prop}
For $\Psi\in L^2(\G\bk \bH)$
\begin{eqnarray}\label{e:spectral}
\Psi(z)&=&\mu_\G(\Psi)+\sum_{k}\<\Psi,\vf_k\> \vf_k(z)\\
\nonumber &&+\sum_\fa\frac{1}{4\pi}\int_{-\infty}^\infty \<\Psi,E_{\G,\fa}(\cdot,\tfrac12+ir)\>E_{\G,\fa}(z,\tfrac12+ir)dr
,
\end{eqnarray}
where $\mu_\G(\Psi)=\frac{1}{v_\G}\int_{\G\bk \bH} \Psi(z)d\mu(z)$ and the first sum is over an orthonormal basis of 
Maass cusp forms.
\end{prop}

The equality is in $L^2(\G\bk\bH)$ and pointwise for $\Psi\in C^\infty_c(\G\bk \bH)$. As a direct consequence, we have the following
\begin{cor}
For $\Psi\in C^\infty_c(\G\bk \bH)$
\begin{eqnarray*}
\|\Psi\|^2&=&|\mu_\G(\Psi)|^2+\sum_{k}|\<\Psi,\vf_k\>|^2+\\
\nonumber &&\sum_\fa\frac{1}{4\pi}\int_{-\infty}^\infty |\<\Psi,E_{\G,\fa}(\cdot,\tfrac12+ir)\>|^2dr.
\end{eqnarray*}
\end{cor}

%
 


\section{Fourier Coefficients}\label{sec:3}

In this section, we derive general bounds for Fourier coefficients of test functions at various cusps. In principle, most of the steps are standard, but we did not find a reference in the literature which carries out each of the necessary calculations, so we give details for the reader's benefit. Another reason for restricting to $\G$ conjugate to $\G_0(p)$ is that the general theory of Fourier coefficients at arbitrary cusps becomes extremely cumbersome (see, e.g., \cite[Theorem 49]{GoldfeldHundleyLee2015}).

First we specify what we mean by the ``Ramanujan conjectures.''
Let $\vf$ be a Hecke-Maass cusp newform for $\G_0(p)$, with Laplace Eigenvalue $\tfrac14+r^2$. 
For $m\neq0$, its $m$th Fourier coefficient (at the cusp $\fa=\infty$) satisfies
$$a_{\vf,\infty}(m,y):=
\int_0^1
\vf(x+iy)e(-mx)dx
=a_{\vf,\infty}(m)
\sqrt{y}K_{
ir
}(2\pi |m| y),
$$
with $K_s(y)$ the Bessel function of the second kind.
The coefficient $ $ further decomposes as
$$
a_{\vf,\infty}(m)=a_{\vf,\infty}(1)\lambda(m),
$$
where
 $\lambda(m)$ the corresponding Hecke eigenvalue. Let $\gt\in[0,1/2]$ be a number so that
\be\label{e:Ramanujan}
|\gl(m)| \ll_\gep |m|^{\gt+\gep}.
\ee
In particular,
$$
\theta=7/64
$$
is known \cite{KimSarnak2003}, while 
 the Ramanujan conjecture predicts that $\theta=0$ holds.

Now, let $\fa$ be a cusp  of a lattice $\G$, 
and
let $\tau_\fa$ denote the corresponding scaling matrix.
Then for any test function $\Psi\in C^\infty_c(\G\bk \bH)$, the 
translated
function
$\Psi^{\tau_\fa}(z):=\Psi(\tau_\fa z)$ is periodic in $x$ with period one and hence has a Fourier expansion
\begin{equation}\label{e:Fourier}
\Psi^{\tau_\fa}(z)=\sum_{m\in \Z}a_{\Psi,\fa}(m;y)e^{2\pi i mx}.
\end{equation}

In \cite[Prop. 2.2]{KelmerKontorovich2017}, we 
proved
that there are constants $0<c_\G<\infty$ and $0<\eta_\G<1$ and some norm $\cS$ (a convex combination of Sobolev norms) such that these coefficients satisfy 
$$|a_{\Psi,\fa}(m,y)|\ll_\G \cS(\Psi)|m|^{c_\G}y^{\eta_\G}$$
uniformly for all $0\neq m\in \Z$ and $y>0$. The argument there was quite soft (using mixing) and applied to any lattice. 
Now we specialize to $\G$ conjugate to $\G_0(p)$
%
to improve the exponents $c_\G$ and $\eta_\G$ above, as well as to have better control on the degree of the Sobolev norm $\cS$. 
Our main result is the following.
\begin{prop}\label{p:Fourier}
Let $\G$ be conjugate to 
$\G_0(p)$.
For any $\Psi\in C^\infty_c(\G\bk \bH)$, for any cusp $\fa$ of $\G$ we have that
\begin{equation}\label{e:a0}
a_{\Psi,\fa}(0,y)=\mu_{\G}(\Psi)+O(\|\Psi\|_2^{3/4}\|\triangle \Psi\|_2^{1/4}y^{1/2}).
\end{equation}
Moreover, for any $m\neq0$, for any $\epsilon>0$ and  any 
$$
\alpha_0>5/3,
$$ 
we have
\begin{equation}\label{e:am}
|a_{\Psi,\fa}(m,y)|\ll_{\alpha_0,\epsilon,p} \cS_{\alpha_0}(\Psi)y^{\tfrac{1}{2}-\epsilon}|m|^{\theta+\epsilon},
\end{equation}
where
$\cS_{\alpha_0}$ is a norm of degree $\alpha_0$.
\end{prop}

In order to prove  \propref{p:Fourier} we consider the spectral decomposition of $\Psi$ into Maass forms and Eisenstein series and bound the Fourier coefficients of each. Explicitly, for the cusp forms we show the following.
\begin{lem}
Let $\vf_k$ be a Hecke-Maass cusp form on $\G_0(p)$ with eigenvalue $\tfrac14+r_k^2$. Then for  any cusp $\fa$, any $m\neq0$, and any $\epsilon>0$, we have
\begin{equation}\label{e:CB}
|a_{\vf_k,\fa}(m,y)|\ll_{\epsilon,p}  (r_k+1)^{-1/3+\epsilon}y^{1/2-\epsilon}|m|^{\theta+\epsilon} 
\end{equation}
\end{lem} 
\begin{proof}
When $\vf$ is an eigenfunction of the Laplacian with eigenvalue $s(1-s)$ we have that
$a_{\vf,\fa}(0,y)$ is a linear combination of $y^s$ and $y^{1-s}$ and for $m\neq 0$ it takes the form
\begin{equation}\label{e:Fourier}
a_{\Psi,\fa}(m;y)=a_{\Psi,\fa}(m)\sqrt{y}K_{s-1/2}(2\pi m y),
\end{equation}
with $K_s(y)$ the Bessel function of the second kind.

Recall that $\G_0(p)$ has two inequivalent cusps, one at $\infty$ and one at $0$. First assume that $\fa$ is equivalent to $\infty$.
Combining \eqref{e:Ramanujan} with  Hoffstein-Lockhart's  \cite{HoffsteinLockhart1994} control on the size of $a_{\vf,\infty}(1)$,
we obtain the bound
\begin{equation}\label{e:HeckeMaassBound}
|a_{\vf,\infty}(m)|\ll_\epsilon (|m|r)^{\epsilon} |m|^\theta e^{\pi r/2}.
\end{equation}
Together with the bound \cite[eq. 4.15]{Strombergsson2004}
for the Bessel function,
\begin{equation}\label{e:Kbound}
|K_{ir}(y)|\ll_\epsilon e^{-\pi r/2}(r+1)^{-1/3+\epsilon}y^{-\epsilon}\min\{1,e^{\pi r/2-y}\},
\end{equation}
we see that \eqref{e:CB} holds in this case. 

For the cusp at $0$ we note that the scaling matrix $\tau_0=\left(\begin{smallmatrix} 0& 1/\sqrt{p}\\ -\sqrt{p} & 0\end{smallmatrix}\right)$ commutes with the Hecke operators $T(n)$ with $(n,p)=1$ and satisfies that 
$\tau_0^{-1}\G_0(p)\tau_0=\G_0(p)$ (see  \cite{Asai1976}). Hence $\vf_k^{\tau_0}$ is also Hecke eigenfunction with the same eigenvalues, and from multiplicity one for new forms  we get that 
$\vf_k^{\tau_0}=c\vf_k$ with some scalar $c$ of modulus $1$. Hence, in absolute value, $|a_{\vf_k,\infty}(m)|=|a_{\vf_k,0}(m)|$ so we have the same bound also for the cusp at $0$.

Finally, the bound for old forms follows directly from the bound for new forms of $\G_1=\SL_2(\Z)$. Explicitly, let $\vf$ be a Hecke-Maass form for $\G_1$ with  Fourier coefficients $a_{\vf}(m,y)$. From this form we get two companion forms $\vf_1(z)=\vf(z)$ and $\vf_2(z)=\vf(pz)$ invariant under $\G_0(p)$. For the cusp at infinity $\vf_1=\vf$ has the same Fourier expansion at infinity as $\vf$. For the second form 
$$\vf_2(z)=\vf(p z)=\sum_m a_{\vf}(m,py)e^{2\pi im px}=\sum_{p|m} a_{\vf}(\tfrac{m}{p},py)e^{2\pi imx},$$
hence  $a_{\vf_2,\infty}(m,y)=a_{\vf}(m/p,py)$ if $p|m$ and is zero otherwise.
The cusp at zero has scaling matrix $\tau_0=k_{\pi/2}\cdot a_p$. Write $\gs=k_{\pi/2}$ so that $\tau_0=\gs a_p$. Since $\vf^\sigma=\vf$ we get that 
$\vf_1^{\tau_0}=\vf^{\sigma a_p}=\vf^{a_p}=\vf_2$, whence $a_{\vf_1,0}=a_{\vf_2,\infty}$. Similarly
$\vf_2^{\tau_0}=\vf^{a_p\sigma a_p}=\vf^\sigma=\vf$ and $a_{\vf_2,0}=a_{\vf_1,\infty}$. Thus the same bound holds also for the cusp at zero.
\end{proof}

Next, we need to bound the Fourier coefficients of Eisenstein series. For each pair of cusps $\fa,\fb$ of $\G$ the Fourier expansion of the Eisenstein series $E_{\G,\fb}$ with respect to the cusp at $\fa$, is given by
$$E^{\tau_\fa}_{\G,\fb}(z,s)=\delta_{\fa,\fb}y^s+\phi_{\fa,\fb}(s)y^{1-s}+\sum_{m\neq 0}a_{\fa,\fb}(s;m,y)e(mx).$$
\begin{lem}
For $\G$ conjugate to $\G_0(p)$ and any two cusps $\fa,\fb$, we have 
\begin{equation}\label{e:EB}
|a_{\fa,\fb}(\tfrac12+ir;m,y)|\ll_{\epsilon} y^{1/2-\epsilon}|m|^\epsilon (1+|r|)^{-1/3+\epsilon}.
\end{equation}
\end{lem}
\begin{proof}
Since $E_{\G,\fa}(z,s)$ is an Eigenfunction with eigenvalue $s(1-s)$ we can write
$$a_{\fa,\fb}(s;m,y)=\phi_{\fa,\fb}(s;m)2\sqrt{y}K_{s-\frac12}(2\pi my).$$
For the full modular group $\G_1=\SL_2(\Z)$ there is just one cusp at $\infty$ and the Fourier coefficients are given explicitly by
$\phi(s)=\frac{\zeta^*(2s-1)}{\zeta^*(2s)}$ and
\begin{equation}
\phi(s,m)=\frac{\tau_{s-1/2}(m)}{\zeta^*(2s)},
\end{equation}
where $\zeta^*(s)=\pi^{-s/2}\zeta(s)\G(s/2)$ is the completed Riemann zeta function and  $\tau_s(m)=\sum_{ab=|m|}(\tfrac ab)^s$ is the divisor function 
\cite[page 67]{Iwaniec1995}. 
In particular, using the Stirling approximation for the $\G$-function and \eqref{e:Kbound} for the Bessel function gives \eqref{e:EB} in this case.
For the congruence groups $\G_0(p)$ the coefficients $\phi_{\fa,\fb}(s,m)$ are given by a similar explicit formula (see \propref{p:phiEis} below), resulting in the same bound.
\end{proof} 

Combining  the above bounds for Fourier coefficients of Maass forms and Eisenstein series we can use the spectral decomposition to bound the Fourier coefficients of any smooth function as follows.
\begin{proof}[Proof of  \propref{p:Fourier}]
First, noting that  $\Psi\in C^\infty_c(\G_0(p)\bk \bH)$ iff $\Psi^\tau\in C^\infty_c(\G\bk \bH)$ and the Fourier coefficients satisfy 
$|a_{\Psi,\fa}(m,y)|=|a_{\Psi^\tau,\fb}(m,y)|$ with $\fb=\tau^{-1}\fa$, we may assume that $\G=\G_0(p)$.

Let $\G=\G_0(p)$ and $\Psi\in C^\infty_c(\G\bk \bH)$. Using the spectral expansion we can write (for any cusp $\fb$) and $m\neq 0$
\begin{eqnarray*}a_{\Psi,\fb}(m,y)&=&\sum_k \langle \Psi,\vf_k\rangle a_{\vf_k,\fb}(m,y)\\
&&+\sum_\fa \frac{1}{2\pi}\int_\R \langle \Psi,E_{\G,\fa}(\cdot,\tfrac{1}{2}+ir)\rangle a_{\fa,\fb}(\tfrac12+ir;m,y)dr
.
\end{eqnarray*}
To bound the contribution of the first sum fix a large parameter $M$ (to be determined later). Applying the bound \eqref{e:CB} to the Fourier coefficients
we get 
\begin{eqnarray*}
\left|\sum_k \langle \Psi,\vf_k\rangle a_{\vf_k,\fb}(m,y)\right|&\ll_\epsilon &
y^{1/2-\epsilon}m^{\theta+\epsilon}\Bigg(\sum_{r_k\leq M} \frac{|\langle \Psi,\vf_k\rangle|}{(r_k+1)^{1/3-\epsilon}}\\
&&
\hskip1in
+\sum_{r_k\geq M} \frac{|\langle \triangle \Psi,\vf_k\rangle|}{r_k^{7/3-\epsilon}}\Bigg)
.
\end{eqnarray*}
Using Cauchy-Schwarz, and Weyl's law stating that $\#\{r_k\leq M\}\ll M^2$ we can bound 
the first sum  by
$$\|\Psi_2\|_2\sqrt{\sum_{r_k\leq M} {(r_k+1)^{-2/3+\epsilon}}}\ll M^{2/3+\epsilon}\cS_{2,0}(\Psi),$$
and the second by
$$\|\triangle \Psi\|_2\sqrt{\sum_{r_k> M} r_k^{-14/3+\epsilon}}\ll M^{-4/3+\epsilon}\cS_{2,2}(\Psi).$$
Choosing $M=\cS_{2,0}^{-1/2}\cS_{2,2}^{1/2}$ we get that 
$$|\sum_k \langle \Psi,\vf_k\rangle a_{\vf_k,\fb}(m,y)|\ll_\epsilon y^{1/2-\epsilon}m^{\theta+\epsilon}\cS_{\tfrac{5}{3}+\epsilon}(\Psi),$$
where the norm $\cS_{5/3+2\epsilon}(\Psi)=\cS_{2,0}(\Psi)^{2/3-\epsilon/2}\cS_{2,2}(\Psi)^{1/3+\epsilon/2}$ is of degree $5/3+\epsilon$.

Next for the Eisenstein integrals for each pair of cusps $\fa,\fb$ use \eqref{e:EB} to get
\begin{eqnarray*}
\left|\int_\R \langle \Psi,E(\cdot,\tfrac{1}{2}+ir)\rangle a_{\fa,\fb}(\tfrac12+ir;m,y)dr\right|&\ll_\epsilon&
 y^{1/2-\epsilon}m^{\epsilon}\Bigg(\int_{|r|\leq M} \frac{|\langle \Psi,E(\cdot,\tfrac{1}{2}+ir)\rangle|}{(1+r)^{1/3-\epsilon}}dr\\
&&
\hskip.5in
+ \int_{|r|> M}\frac{|\langle \triangle \Psi,E(\cdot,\tfrac{1}{2}+ir)\rangle|}{r^{7/3-\epsilon}}\Bigg)dr
.
\end{eqnarray*}
As before we can use Cauchy-Schwarz to bound the first integral by 
$O_\epsilon(\cS_{2,0}(\Psi)M^{1/6+\epsilon})$ and the second by $O_\epsilon(\cS_{2,2}(\Psi)M^{-11/6+\epsilon})$ so taking $M=\cS_{2,0}^{-1/2}\cS_{2,2}^{1/2}$ the whole integral is bounded by
$$\int_\R \langle \Psi,E(\cdot,\tfrac{1}{2}+ir)\rangle a_r(m,y)dr\ll y^{1/2-\epsilon}m^{\epsilon}\cS_{7/6+\epsilon}(\Psi)
,
$$
where 
$$\cS_{7/6+\epsilon}(\Psi)=\cS_{2,0}(\Psi)^{11/1-\epsilon/2}\cS_{2,2}^{1/12+\epsilon/2}
,$$
is of degree $7/6+\epsilon$.

Collecting the contributions of all cusps we get that 
$$|a_\Psi(m,y)|\ll y^{1/2-\epsilon}m^{\epsilon}\left(m^\theta\cS_{\tfrac{5}{3}+2\epsilon}(\Psi)+\cS_{14/12+2\epsilon}(\Psi)\right).$$
Taking $\epsilon$ sufficiently small so that $\tfrac{5}{3}+2\epsilon=\alpha_0$ concludes the proof for the non trivial coefficients.

For the trivial ($m=0$) coefficient, again using the spectral expansion, the only contribution comes from the constant function (giving the main term) and the Eisenstein integrals. We thus need to bound 
for each pair of cusps
\begin{eqnarray*}
\int_\R \<\Psi,E_{\G,\fa}(\tfrac12+ir,\cdot)\>(y^{1/2+ir}+\phi_{\fa,\fb}(\tfrac12+ir)y^{1/2-ir})dr\\\nonumber
 \ll y^{1/2}\int_\R|\<\Psi,E_{\G,\fa}(\tfrac12+ir,\cdot)\>|dr.
\end{eqnarray*}
As before we separate 
\begin{eqnarray*}\label{e:EisensteinInt}
\int_\R|\<\Psi,E_\G(\tfrac12+ir,\cdot)\>|dr =
\int_{|r|<M}|\<\Psi,E_{\G,\fa}(\tfrac12+ir,\cdot)\>|dr\\\nonumber +\int_{|r|>M}|\<\Psi,E_{\G,\fa}(\tfrac12+ir,\cdot)\>|dr.
\end{eqnarray*}
Using Cauchy Schwarz, we can bound the first integral by
\begin{eqnarray*}
\int_{|r|<M}|\<\Psi,E_{\G,\fa}(\tfrac12+ir,\cdot)\>|dr&\leq& \sqrt{2M}\sqrt{\int_{\R}|\<\Psi,E_{\G,\fa}(\tfrac12+ir,\cdot)\>|^2dr}\\
&\ll& \sqrt{M}\|\Psi\|_2.
\end{eqnarray*}
and
\begin{eqnarray*}
\int_{|r|>M}|\<\Psi,E_{\G,\fa}(\tfrac12+ir,\cdot)\>|dr&=&
\int_{|r|>M}\frac{|\<\triangle f,E_{\G,\fa}(\tfrac12+ir,\cdot)\>|}{1+r^2}dr\\
&\ll& \|\triangle \Psi\|_2 M^{-3/2},
\end{eqnarray*}
We thus get that
$$\int_\R|\<\Psi,E_{\G,\fa}(\tfrac12+ir,\cdot)\>|dr\ll \sqrt{M}\|\Psi\|_2+ M^{-3/2}\|\triangle \Psi\|_2,$$
and the optimal choice of $M=\sqrt{\frac{\|\triangle f\|_2}{\|f\|_2}}$ 
gives the desired bound.
\end{proof}

\subsection{Additonal estimates}
We end this section with several estimates which follow from our bounds on Fourier coefficients. 
First, as consequence of the estimate for $a_{\Psi,\fa}(0,y)$ we get the following useful estimate that we record here for future use.
\begin{cor}\label{c:L2bound}
Let $\G$ be conjugate to $\G_0(p)$. For any cusp $\fa$ with scaling matrix $\tau_\fa$, for any $\psi\in C^\infty_c(\G\bk \bH)$ we have the bound
\begin{equation}\label{e:L2bound}
\left(\int_0^1|\psi^{\tau_\fa}(x+iy)|^2dx\right)^{1/2}\ll \cS_1(\psi)+\cS_2(\psi)y^{1/4},
\end{equation}
with $\cS_1$ and $\cS_2$ are suitable norms of degree $1$ and $2$ respectively.
\end{cor}
\begin{rem}
When $y$ is small and $\psi$ approximates a bump function this is an improvement over the trivial bound of $\cS_{\infty,0}(\psi)$ which is a norm of degree $2$.
\end{rem}
\begin{proof}
Using  \eqref{e:a0} with the test function $\Psi(z)=|\psi^{\tau_\fa}(z)|^2$, we get that
$$\int_0^1|\psi^{\tau_\fa}(x+iy)|^2dx=\mu(\Psi)+O(\|\Psi\|_2^{3/4}\|\triangle \Psi\|_2^{1/4}y^{1/2})).$$
For $\Psi=|\psi|^2$ we have
$$\mu(\Psi)=|\cS_{2,0}(\psi)|^2,\; \|\Psi\|_2= |\cS_{4,0}(\psi)|^2,\;\mbox{ and } \|\triangle \Psi\|_2\ll |\cS_{4,2}(\psi)|^2.$$
Define the norms $\cS_1(\psi)=\cS_{2,0}(\psi)$ and
$$\cS_2(\psi):=\cS_{4,0}(\psi)^{3/4} \cS_{4,2}(\psi)^{1/4}.$$
Clearly $\cS_1$ is of degree $1$ and $\cS_2$ is of  degree $2$. 
Finally taking a square root gives the result.
\end{proof}

Combining \corref{c:L2bound} and \propref{p:Fourier} we obtain another estimate that we will need.
\begin{prop}\label{l:AveFourier}
Let $\G$ be conjugate to $\G_0(p)$ and $\Psi\in \cC_c^\infty(\G\bk \bH)$. Then for any $\alpha_1>\frac{5+6\theta}{3+6\theta}$ and any $\eta_1<\frac{1}{2+4\theta}$
$$\sum_{m\neq 0}\frac{|a_{\Psi,\fa}(m,y)|}{m}\ll_{\alpha_1,\eta_1} \cS_{\alpha_1}(\Psi)y^{\eta_1}+\cS_{\alpha_1+1}(\Psi)y^{\eta_1+1/4},$$
with $\cS_{\alpha_1}$ and $\cS_{\alpha_1+1}$ appropriate norms of degrees $\alpha_1$ and $\alpha_1+1$ respectively.
\end{prop}
\begin{proof}
Replacing $\Psi$ with $\Psi^{\tau_\fa}$ we may assume that $\fa=\infty$ and $\G$ has cusp of width one. 
We can find $\epsilon>0$ and $\alpha_0>5/3$, sufficiently small so that $\alpha_1=\frac{\alpha_0+2\theta+2\epsilon}{1+2\theta+2\epsilon}$ and that
$\eta_1=\frac{1+2\epsilon}{1+2\theta-2\epsilon}$.
Fix a large parameter $M$ and separate the sum to 
$$\sum_{m\neq 0}\frac{|a_{\Psi,\fa}(m,y)|}{m}=\sum_{|m|\leq M}\frac{|a_{\Psi,\fa}(m,y)|}{m}+\sum_{m>M}\frac{|a_{\Psi,\fa}(m,y)|}{m}$$
For the first sum, using  \propref{p:Fourier} we get that 
   $$\sum_{|m|\leq M}\frac{|a_{\Psi,\fa}(m,y)|}{m}\ll \cS_{\alpha_0}(\Psi)y^{1/2-\epsilon}M^{\theta+\epsilon}.$$
For the second sum using Cauchy Schwarz, followed by Parseval, and then applying  \corref{c:L2bound} we get
\begin{eqnarray*}
\sum_{m>M}\frac{|a_{\Psi,\fa}(m,y)|}{m}&\leq& \frac{1}{\sqrt{M}}\left(\sum_{|m|>M}|a_{\Psi,\fa}(m,y)|^2\right)^{1/2}\\
&\leq& \frac{1}{\sqrt{M}}\left(\int_0^1 |\Psi(x+iy)|^2dx \right)^{1/2}\\
&\ll& \frac{\cS_1(\Psi)+\cS_2(\Psi)y^{1/4}}{\sqrt{M}}.
\end{eqnarray*}
Fix an optimal choice of $M=(\frac{\cS_1(\Psi)}{\cS_2(\Psi)})^{\frac{2}{1+2\theta+\epsilon}}y^{\frac{2\epsilon-1}{1+2\theta+2\epsilon}}$ to get that
$$\sum_{m\neq 0}\frac{|a_{\Psi,\fa}(m,y)|}{m}\ll \cS_{\alpha_1}(\Psi)y^{\eta_1}+\cS_{\alpha_1+1}(\Psi)y^{\eta_1+1/4},$$
with the norms $\cS_{\alpha_1}(\Psi)=\cS_1(\Psi)^{\frac{2\theta+2\epsilon}{1+2\theta+2\epsilon}}\cS_{\alpha_0}(\Psi)^{\frac{1}{1+2\theta+2\epsilon}}$ and $\cS_{\alpha_1+1}(\Psi)= \tfrac{\cS_{\alpha_1}(\Psi)\cS_2(\Psi)}{\cS_1(\Psi)}$.
\end{proof}


\section{Equidistribution of shears}\label{sec:edShears}
We now  use our results on the Fourier coefficients from the previous section to improve the error term in the equidistribution result of \cite[Theorem 1.3]{KelmerKontorovich2017}.
In addition to improving the error term we also take care to make the dependance of the error on width of the cusp explicit, as this will be needed for our application.
We will show the following
\begin{thm}\label{t:Equidistribution1}
Let $\G$ be a conjugate of $\G_0(p)$ with a cusp at $\infty$ of width $\omega\geq 1$.
For any $\tfrac{7+12\theta}{3+6\theta}<\alpha<3$ and any $\tfrac14<\eta_1<\tfrac{1}{2+4\theta}$, there are norms $\cS_{\alpha},\cS_{\alpha_2}$ of degrees $\alpha$ and $\alpha_2=1+\alpha+\tfrac{3-\alpha}{4\eta}$ respectively , so that for any $\Psi\in C_c^\infty(\G\bk \bH)$ and any $T\geq \omega^{1/\eta_1}$
\begin{eqnarray*}
\int_{\frac{1}{\sqrt{1+T^2}}}^\infty \Psi(yT+iy)\frac{dy}{y}&=&\mu_{\G}(\Psi)\log(T\omega )+\langle \cK_{\G,\infty},\Psi\rangle\\
&&+O_{\alpha,\eta_1}\left(\frac{\cS_\alpha(\Psi)\omega^{\frac12}}{T^{\frac{\eta_1}{2}}}+\frac{\cS_{\alpha_2}(\Psi)\omega^{\frac{1}{4}}}{T^{\frac{\eta_1}{2}+\frac{1}{8}}}\right)
.
\end{eqnarray*}
\end{thm}

As in \cite{KelmerKontorovich2017}, the proof of  \thmref{t:Equidistribution1} splits into two parts. The first gives equidistribution in the strip (compare to \cite[Theorem 3.3]{KelmerKontorovich2017}) which in our setting is

\begin{prop}\label{p:Strip}
Under the same conditions and notations as in  \thmref{t:Equidistribution1} we have
\begin{eqnarray*}
\int_{\frac{1}{\sqrt{1+T^2}}}^\infty \Psi(yT+iy)\frac{dy}{y}&=&\frac{1}{\omega}\int_0^\omega\int_{1/T}^\infty \Psi(x+iy)\frac{dydx}{y}\\
&&+O_{\alpha,\eta_1}\left(\frac{\cS_\alpha(\Psi)\omega^{\frac12}}{T^{\frac{\eta_1}{2}}}+\frac{\cS_{\alpha_2}(\Psi)\omega^{\frac{1}{4}}}{T^{\frac{\eta_1}{2}+\frac{1}{8}}}\right).
\end{eqnarray*}
\end{prop}
The second step  uses the theory of Eisenstein series to estimate the strip average by the Eisenstein distribution (see \cite[Theorem 3.6]{KelmerKontorovich2017}). 
\begin{prop}\label{p:Eisenstein}
Under the same conditions and notations as in  \thmref{t:Equidistribution1} we have, 
\begin{eqnarray*}
\frac{1}{\omega}\int_0^\omega\int_{1/T}^\infty \Psi(x+iy)\frac{dydx}{y}&=&\mu_{\G}(\Psi)\log(T\omega)+\langle  \cK_{\G,\infty},\Psi\rangle+O\left(\frac{\cS_{2,0}(\Psi)}{\sqrt{\omega T}}\right)
.
\end{eqnarray*}
\end{prop}

The proof of  \propref{p:Eisenstein} follows exactly as that of \cite[Theorem 3.6]{KelmerKontorovich2017} without any changes. The only difference is the absence of residual spectrum in this case, and the fact that we are considering $K$-invariant test functions, allowing us to use the Sobolev norm $\cS_{2,0}(\Psi)$ instead of $\cS_{2,1}(\Psi)$.  \thmref{t:Equidistribution1} follows from these two propositions after noting that the error term $O(\frac{\cS_{2,0}(\Psi)}{\sqrt{\omega T}})$ is subsumed by the other terms.
We thus devote the rest of this section to the proof of  \propref{p:Strip}, taking advantage of the assumption that
 $\G$ is a conjugate of the congruence group $\G_0(p)$.

\subsection{Equidistribution in the strip}
In order to prove \propref{p:Strip} write $\Psi(z)=a_{\Psi,\infty}(0,\frac{y}{\omega})+\Psi^\perp(z)$ with $\Psi^\perp$ the projection of $\Psi$ to the space orthogonal to the constant functions. The main term will come from the constant term, and we will bound the remaining integral of $\Psi^\perp$. 
To do this we prove the following lemma.
\begin{lem}\label{l:unit}
Let $\G$ be conjugate to $\G_0(p)$ with a cusp at $\infty$ of width $\omega$.
For any $\alpha>\frac{7+12\theta}{3+6\theta}$ and any $\eta_1<\frac{1}{2+4\theta}$, there are norms $\cS_{\alpha},\cS_{\alpha+1}$ of orders $\alpha$ and $\alpha+1$ respectively, so that for any $\Psi\in C_c^\infty(\G\bk \bH)$, for any $C\geq 1$,  $\beta\in(0,\tfrac{1-\eta_1}{3})$ and any $1\leq k<CK_0$ we have
$$\left|\int_k^{k+1}\Psi^\perp(y+i\frac{y}{T})\frac{dy}{y}\right|\ll C^{\frac{1+\eta_1}{2}}\left(\cS_\alpha(\Psi)\frac{\omega^{1-\eta_1-\beta}}{k^{\frac{3-\eta_1}{2}}T^{\frac{\eta_1}{2}}}+\cS_{\alpha+1}\frac{\omega^{\frac{3}{4}-\eta_1-\beta}}{k^{\frac{5}{4}-\frac{\eta_1}{2}} T^{\frac{1}{4}+\frac{\eta_1}{2}}}\right).$$
Here
$$K_0=K_0(\Psi,T)=\big(\tfrac{\cS_{\infty,1}(\Psi)^2}{\cS_{\alpha}(\Psi)^2}\omega^{-2\beta}T^{\eta_1}\big)^{\frac{1}{\eta_1+1}}.$$
\end{lem}
\begin{proof}
Fix a large parameter $N\in\N$ to be determined later and write 
$$\int_k^{k+1}\Psi^\perp(y+i\frac{y}{T})\frac{dy}{y}=\sum_{j=0}^{N-1}\int_{t_j}^{t_{j+1}}\Psi^\perp(y+i\frac{y}{T})\frac{dy}{y},$$
with $t_j=\tfrac{Nk+j}{N}$. For $t_j\leq y\leq t_{j+1}$ we can approximate 
$$\Psi^\perp(y+i\frac{y}{T})=\Psi^\perp(y+i\frac{t_j}{T})+O(\frac{\cS_{\infty,1}(\Psi)}{Nk+j}),$$
where we used that, for the hyperbolic distance, 
$$d(y+i\frac{y}{T},y+i\frac{t_j}{T})=|\log(y/t_j)|\leq \log(t_{j+1}/t_j)\leq (Nk+j)^{-1}.$$
\begin{rmk}\label{rmk:error}
It is here that we correct the error in the proof of \cite[Lemma 3.6]{KelmerKontorovich2017}.
\end{rmk}

Plugging this in gives 
$$\int_{t_j}^{t_{j+1}}\Psi^\perp(y+i\frac{y}{T})\frac{dy}{y}=\int_{t_j}^{t_{j+1}}\Psi^\perp(y+i\frac{t_j}{T})\frac{dy}{y}+O\left(\frac{\cS_{\infty,1}(\Psi)}{(Nk+j)^2}\right).$$
Now for the first term, expanding $\Psi^\perp$ into its Fourier series, using integration by parts to bound $
|\int_{t_j}^{t_{j+1}}e^{2\pi i \frac{my}{\omega}}\frac{dy}{y}|\ll \frac{\omega}{mk}$, and using \propref{l:AveFourier} with $\eta_1$ and $\alpha_1=2\alpha-3$ we can bound

\begin{eqnarray*}
\left|\int_{t_j}^{t_{j+1}}\Psi^\perp(y+i\frac{t_j}{T})\frac{dy}{y}\right|&=&\left|\sum_{m\neq 0}a_{\Psi,\infty}(m,\frac{t_j}{\omega T})\int_{t_j}^{t_{j+1}}e^{2\pi i \frac{my}{\omega}}\frac{dy}{y}\right|\\
&\ll& \frac{\omega}{k} \sum_{m\neq 0}\frac{|a_{\Psi,\infty}(m,\frac{t_j}{\omega T})|}{m}\\
&\ll&  \frac{\omega}{k} (\cS_{\alpha_1}(\Psi)(\tfrac{k}{\omega T})^{\eta_1}+\cS_{\alpha_1+1}(\Psi)(\tfrac{k}{\omega T})^{\eta_1+1/4})\\
&\ll&  
\frac{\omega^{1-\eta_1}}{k^{1-\eta_1}} \frac{\cS_{\alpha_1}(\Psi)}{T^{\eta_1}}+\frac{\omega^{3/4-\eta_1}}{k^{3/4-\eta_1}}\frac{\cS_{\alpha_1+1}(\Psi)}{T^{\eta_1+1/4}}.
\end{eqnarray*}
Summing over $0\leq j<N$ we get that 
$$\left|\int_k^{k+1}\Psi^\perp(y+i\frac{y}{T})\frac{dy}{y}\right|\ll N\frac{\omega^{1-\eta_1}}{k^{1-\eta_1}} \frac{\cS_{\alpha_1}(\Psi)}{T^{\eta_1}}+N\frac{\omega^{3/4-\eta_1}}{k^{3/4-\eta_1}}\frac{\cS_{\alpha_1+1}(\Psi)}{T^{\eta_1+1/4}}+\frac{\cS_{\infty,1}(\Psi)}{Nk^2}.$$

First, for $k\leq K_0$ we can take
$$N=N_0(k)=[\sqrt{\tfrac{\cS_{\infty,1}(\Psi)}{\cS_{\alpha_1}(\Psi)}}\tfrac{T^{\frac{\eta_1}{2}}}{\omega^\beta k^{\frac{1+\eta_1}{2}}}]\geq 1,$$
and defining the norms 
$$
\cS_{\alpha}(\Psi)=\sqrt{\cS_{\alpha_1}(\Psi)\cS_{\infty,1}(\Psi)},\quad \cS_{\alpha+1}(\Psi)=\sqrt{\frac{\cS_{\infty,1}(\Psi)\cS_{\alpha_1+1}(\Psi)^2}{\cS_{\alpha_1}(\Psi)}}.$$
(that are indeed of degrees $\alpha$ and $\alpha+1$), we get that 
$$\left|\int_k^{k+1}\Psi^\perp(y+i\frac{y}{T})\frac{dy}{y}\right|\ll \cS_\alpha(\Psi)\frac{\omega^{1-\eta_1-\beta}}{k^{\frac{3}{2}-\frac{\eta_1}{2}}T^{\frac{\eta_1}{2}}}+\cS_{\alpha+1}(\Psi)\frac{\omega^{\frac{3}{4}-\eta_1-\beta}}{k^{\frac{5}{4}-\frac{\eta_1}{2}} T^{\frac{1}{4}+\frac{\eta_1}{2}}}.$$

Next for $K_0\leq k<CK_0$ we take instead   $N=[C^{\frac{1+\eta_1}{2}}N_0]\geq 1$ to get that 
$$\left|\int_k^{k+1}\Psi^\perp(y+i\frac{y}{T})\frac{dy}{y}\right|\ll C^{\frac{1+\eta_1}{2}}\cS_\alpha(\Psi)\frac{\omega^{1-\eta_1-\beta}}{k^{\frac{3}{2}-\frac{\eta_1}{2}}T^{\frac{\eta_1}{2}}}+C^{\frac{1+\eta_1}{2}}\cS_{\alpha+1}(\Psi)\frac{\omega^{\frac{3}{4}-\eta_1-\beta}}{k^{\frac{5}{4}-\frac{\eta_1}{2}} T^{\frac{1}{4}+\frac{\eta_1}{2}}}.$$
\end{proof}

Using this estimate it is possible to evaluate $\int_1^{K_0}\Psi^\perp(y+i\frac{y}{T})\frac{dy}{y}$. We now show how to boot strap this to extend the range all the way up to $K_1=K_0^{\frac{1+\eta}{2\eta}}$.
\begin{prop}\label{p:bootstrap}
With the same assumptions and notations as in \lemref{l:unit}, let $K_1=K_1(\Psi,T)=K_0(\Psi,T)^{\frac{1+\eta_1}{2\eta_1}}$. Then there is a constant $c\in (1/e,1)$ such that 
$$\int_1^{cK_1}\Psi(y+i\frac{y}{T})\frac{dy}{y}\ll \log\log(K_0)\left(\frac{\cS_\alpha(\Psi)\omega^{1-\eta_1-\beta}}{T^{\frac{\eta_1}{2}}}+\frac{\cS_{\alpha_2}(\Psi) \omega^{\frac{3}{4}-\eta_1-\beta-\frac{\beta}{4\eta_1}}}{T^{\frac{\eta_1}{2}+\frac{1}{8}}}   \right),$$
where $\cS_\alpha$ is as above and $\cS_{\alpha_2}$ is a norm of degree $\alpha_2=\frac{4\eta+3}{4\eta}+\alpha\frac{4\eta-1}{4\eta}$.
\end{prop}
\begin{proof}
As a first step, a simple application of  \lemref{l:unit} with $C=1$, noting that for both terms the power of $k$ in the denominator is greater than one gives
\begin{eqnarray*}\int_1^{K_0}\Psi(y+i\frac{y}{T})\frac{dy}{y}&=&\sum_{k=1}^{K_0-1}\int_k^{k+1}\Psi(y+i\frac{y}{T})\frac{dy}{y}\\
&\ll& \cS_\alpha(\Psi)\frac{\omega^{1-\eta_1-\beta}}{T^{\frac{\eta_1}{2}}}+\cS_{\alpha+1}(\Psi)\frac{\omega^{\frac{3}{4}-\eta_1-\beta}}{ T^{\frac{1}{4}+\frac{\eta_1}{2}}}.
\end{eqnarray*}

Next let $s_\ell$ denote the partial sums of geometric series $s_\ell=\sum_{j=0}^\ell (\tfrac{1-\eta_1}{1+\eta_1})^j$ converging to $s_\infty=\frac{1+\eta_1}{2\eta_1}$ and let $C_\ell=K_0^{s_\ell-1}$. Applying  \lemref{l:unit} with $C=C_\ell$ we get the bound
\begin{eqnarray*}
\int_{C_{\ell-1}K_0}^{C_{\ell}K_0}\Psi(y+i\frac{y}{T})\frac{dy}{y}&=&\sum_{k=C_{\ell-1}K_0}^{C_\ell K_0-1}|\int_k^{k+1}\Psi(y+i\frac{y}{T})\frac{dy}{y}|\\
&\ll& \frac{C_{\ell}^{\frac{1+\eta_1}{2}}\cS_\alpha(\Psi)}{(C_{\ell-1}K_0)^{\frac{1-\eta_1}{2}}}\frac{\omega^{1-\eta_1-\beta}}{T^{\frac{\eta_1}{2}}}+\frac{C_{\ell}^{\frac{1+\eta_1}{2}}\cS_{\alpha+1}(\Psi)}{(C_{\ell-1}K_0)^{\frac{1-2\eta}{4}}}\frac{\omega^{\frac{3}{4}-\eta_1-\beta}}{ T^{\frac{1}{4}+\frac{\eta_1}{2}}}.
\end{eqnarray*}
From our choice of the constants we have that $C_\ell=[K_0 C_{\ell-1}]^{\frac{1-\eta}{1+\eta}}$ so that $\frac{C_{\ell}^{1+\eta_1}}{(C_{\ell-1}K_0)^{1-\eta_1}}=1$ and
$\frac{C_{\ell}^{1+\eta_1}}{(C_{\ell-1}K_0)^{\frac{1-2\eta}{2}}}=(C_{\ell-1}K_0)^{1/2}$, hence
\begin{eqnarray*}
\int_{C_{\ell-1}K_0}^{C_{\ell}K_0}\Psi(y+i\frac{y}{T})\frac{dy}{y}
\ll\cS_\alpha(\Psi)\frac{\omega^{1-\eta_1-\beta}}{T^{\frac{\eta_1}{2}}}+[C_{\ell-1}K_0]^{1/4}\cS_{\alpha+1}(\Psi)\frac{\omega^{\frac{3}{4}-\eta_1-\beta}}{ T^{\frac{1}{4}+\frac{\eta_1}{2}}}.
\end{eqnarray*}
Now bounding $C_{\ell-1}K_0\leq K_0^{\frac{1+\eta_1}{2\eta_1}}=K_1$ and plugging in 
$K_0=\big(\tfrac{\cS_{\infty,1}(\Psi)^2}{\cS_{\alpha}(\Psi)^2}\omega^{-2\beta}T^\eta_1\big)^{\frac{1}{\eta_1+1}}$ we can bound the second term by
$$K_0^{\frac{1+\eta_1}{8\eta_1}}\cS_{\alpha+1}(\Psi)\frac{\omega^{\frac{3}{4}-\eta_1}}{ T^{\frac{1}{4}+\frac{\eta_1}{2}}}=
\frac{\cS_{\alpha_2}(\Psi)\omega^{\frac{3}{4}-\eta_1-\beta-\frac{\beta}{4\eta_1}}}{ T^{\frac{1}{8}+\frac{\eta_1}{2}}},$$
with 
$$\cS_{\alpha_2}(\Psi)=\cS_{\alpha+1}(\Psi)\cS_{\infty,1}(\Psi)^{\frac{1}{4\eta_1}}\cS_{\alpha}(\Psi)^{\frac{-1}{4\eta_1}}.$$
We thus get that for each $\ell\geq 1$ we have
\begin{eqnarray*}
\int_{C_{\ell-1}K_0}^{C_{\ell}K_0}\Psi(y+i\frac{y}{T})\frac{dy}{y}
\ll\cS_\alpha(\Psi)\frac{\omega^{1-\eta_1-\beta}}{T^{\frac{\eta_1}{2}}}+\frac{\cS_{\alpha_2}(\Psi)\omega^{\frac{3}{4}-\eta_1-\beta-\frac{\beta}{4\eta_1}} }{T^{\frac{1}{8}+\frac{\eta_1}{2}}}.
\end{eqnarray*}

Finally, taking $\ell= [\frac{\log\log(K_0)}{\log(\tfrac{1+\eta}{1-\eta})}]+1$ we get that 
$$s_\ell\geq \tfrac{1+\eta_1}{2\eta_1}- (\tfrac{1-\eta_1}{1+\eta_1})^\ell\geq  \tfrac{1+\eta_1}{2\eta_1}-\tfrac{1}{\log(K_0)},$$ 
so that 
$e^{-1} K_1\leq C_\ell K_0\leq K_1$, hence, $C_\ell K_0=cK_1$ for some $c>1/e$.
Summing up these $\ell=O(\log\log(K_0))$ terms concludes the proof.
\end{proof}

For large values of $k>K_1$ this estimate is no longer optimal, and instead we will use the following alternative bound.
\begin{lem}\label{l:period}
Let $\G$ is conjugate of $\G_0(p)$ with a cusp at $\infty$ of width $\omega$. For any $\alpha_1>\frac{5+6\theta}{3+6\theta}$ and $\eta_1<\frac{1}{2+4\theta}$ there are norms $\cS_{\alpha_1},\cS_{\alpha_1+1}$  such that for any 
$\Psi\in C^\infty_c(\G\bk \bH)$, for any $k\geq 1$ we have
$$\int_{\omega k}^{\omega(k+1)}\Psi^\perp(y+i\frac{y}{T})\frac{dy}{y}\ll \frac{\cS_{\alpha_1}(\Psi)}{k^{2-\eta_1}T^{\eta_1}}+\frac{\cS_{\alpha_1+1}(\Psi)}{k^{7/4-\eta_1}T^{\eta_1+1/4}}+\frac{\cS_{\infty,1}(\Psi)}{k^2}.$$
\end{lem}
\begin{proof}
For $k\omega\leq y\leq (k+1)\omega$ we can estimate 
$$
\Psi(y+i\frac{y}{T})=\Psi(y+i\frac{k\omega}{T})+O(\frac{\cS_{\infty,1}(\Psi)}{k})
,
$$ 
so that 
$$\int_{\omega k}^{\omega(k+1)}\Psi^\perp(y+i\frac{y}{T})\frac{dy}{y}=\int_{\omega k}^{\omega(k+1)}\Psi^\perp(y+i\frac{\omega k}{T})\frac{dy}{y}+O(\frac{\cS_{\infty,1}(\Psi)}{k^2}).$$
Expanding the first term in Fourier series 
$$\int_{\omega k}^{\omega(k+1)}\Psi^\perp(y+i\frac{\omega k}{T})\frac{dy}{y}=\sum_{m\neq 0}a_{\Psi}(m,\frac{k}{T})\int_{\omega k}^{\omega(k+1)}e^{\frac{2\pi i my}{\omega}}\frac{dy}{y}.$$
Integrating by parts we can bound 
$$\left|\int_{\omega k}^{\omega(k+1)}e^{\frac{2\pi i my}{\omega}}\frac{dy}{y}\right|=\left|\int_{k}^{k+1}e^{2\pi i my}\frac{dy}{y}\right|
\leq \frac{1}{mk^2},$$
and using \propref{l:AveFourier} we bound
\begin{eqnarray*}
\int_{\omega k}^{\omega(k+1)}\Psi^\perp(y+i\frac{\omega k}{T})\frac{dy}{y}&=&\sum_{m\neq 0}a_{\Psi}(m,\frac{k}{T})\int_{\omega k}^{\omega(k+1)}e^{\frac{2\pi i my}{\omega}}\frac{dy}{y}\\
&\leq &\frac{1}{k^2}\sum_{m\neq 0}\frac{|a_{\Psi}(m,\frac{k}{T})|}{m}\\
&\ll &  \frac{\cS_{\alpha_1}(\Psi)}{k^{2-\eta_1}T^{\eta_1}}+\frac{\cS_{\alpha_1+1}(\Psi)}{k^{7/4-\eta_1}T^{\eta_1+1/4}}.
\end{eqnarray*}
concluding the proof.
\end{proof}

 We can now give the 
 \begin{proof}[Proof of Proposition \propref{p:Strip}]
Noting that  
$$\int_{\frac{T}{\sqrt{T^2+1}}}^\infty \Psi(y+i\tfrac{y}{T})\frac{dy}{y}=\int_{1}^\infty \Psi(y+i\frac{y}{T})\frac{dy}{y}+O\left(\frac{\cS_{\infty,0}(\Psi)}{T^2}\right),$$
and that 
$$\int_{1}^\infty a_{\Psi,\infty}(0,\frac{y}{\omega T})\frac{dy}{y}=\frac{1}{\omega}\int_0^\omega\int_{1/T}^\infty \Psi(x+iy)\frac{dydx}{y},$$
we just need to bound $\int_{1}^\infty \Psi^\perp(y+i\frac{y}{T})\frac{dy}{y}$. 

Let $\beta=\tfrac{1}{2}-\eta_1$ and let $K_0=K_0(\Psi,T)=\big(\tfrac{\cS_{\infty,1}(\Psi)^2}{\cS_{\alpha}(\Psi)^2}\omega^{-2\beta}T^{\eta_1}\big)^{\frac{1}{\eta_1+1}}$ be as in \lemref{l:unit} and 
$K_1=K_0^{\frac{1+\eta_1}{2\eta_1}}=(\tfrac{\cS_{\infty,1}(\Psi)}{\cS_{\alpha}(\Psi)})^{\frac{1}{\eta_1}}\omega^{-\frac{\beta}{\eta_1}}T^{\frac12}$ be as in \propref{p:bootstrap}. Then
\begin{equation*}
\int_1^{cK_1}\Psi^\perp(y+i\frac{y}{T})\frac{dy}{y}\ll \log\log(K_0)\left(\frac{\cS_\alpha(\Psi)\omega^{\frac12}}{T^{\frac{\eta_1}{2}}}+\frac{\cS_{\alpha_2}(\Psi)\omega^{\frac{1}{4}-\frac{1-2\eta}{8\eta}}}{T^{\frac{\eta_1}{2}+\frac{1}{8}}}   \right).
\end{equation*}
Next, 
using Lemma \lemref{l:period} for $\tilde\Psi=\Psi^{a_{\omega}}$ having period one, we can bound
\begin{eqnarray*}\int_{cK_1}^{\infty}\Psi^\perp(y+i\frac{y}{T})\frac{dy}{y}&=& \int_{\frac{cK_1}{\omega}}^{\infty}\tilde{\Psi}^\perp(y+i\frac{y}{T})\frac{dy}{y}\\
&=&\sum_{k>cK_1/\omega} \int_k^{k+1}\tilde{\Psi}^\perp(y+i\frac{y}{T})\frac{dy}{y}\\
&\ll& \frac{\cS_{\alpha_1}(\Psi)}{T^{\eta_1}}+\frac{\cS_{\alpha_1+1}(\Psi)}{T^{\eta_1+1/4}}+\frac{\omega\cS_{\infty,1}(\Psi)}{K_1}.
\end{eqnarray*}
The first two terms are clearly bounded by the similar terms appearing above and plugging in the value of $K_1$, the third term is 
$$\frac{\omega\cS_{\infty,1}(\Psi)}{K_1}= \cS_{\infty,1}(\Psi)^{1-\frac{1}{\eta_1}}\cS_{\alpha}(\Psi)^{\frac{1}{\eta_1}}\big(\frac{\omega^{\eta_1+\beta}}{T^{\frac{\eta_1}{2}}}\big)^{\frac{1}{\eta_1}}.$$
Notice that $\cS_{\infty,1}(\Psi)^{1-\frac{1}{\eta_1}}\cS_{\alpha}(\Psi)^{\frac{1}{\eta_1}}$ is a norm of degree $\frac{\alpha}{\eta_1}-3(\frac{1}{\eta_1}-1)$ (which is smaller than $\alpha$ as long as $\alpha<3$),
and with our choice of $\beta=\tfrac{1}{2}-\eta$ we see that 
$$\left(\frac{\omega^{\eta_1+\beta}}{T^{\frac{\eta_1}{2}}}\right)^{\frac{1}{\eta_1}}=\left(\frac{\omega^{\frac12}}{T^{\frac{\eta_1}{2}}}\right)^{\frac{1}{\eta_1}}\leq \frac{\omega^{\frac12}}{T^{\frac{\eta_1}{2}}}.$$
Hence, after perhaps replacing $\cS_\alpha$ by a different norm of degree $\alpha$ we get that 
$$\frac{\omega\cS_{\infty,1}(\Psi)}{K_1}\leq \frac{\cS_\alpha(\Psi)\omega^{\frac12}}{T^{\frac{\eta_1}{2}}},$$
whence
\begin{equation*}
\int_1^{\infty}\Psi^\perp(y+i\frac{y}{T})\frac{dy}{y}\ll \log\log(K_0)\left(\frac{\cS_\alpha(\Psi)\omega^{\frac12}}{T^{\frac{\eta_1}{2}}}+\frac{\cS_{\alpha_2}(\Psi)\omega^{\frac{1}{4}}}{T^{\frac{\eta_1}{2}+\frac{1}{8}}}   \right).
\end{equation*}
Finally, replacing the exponent $\eta_1$ with a slightly smaller exponent we may remove the $\log\log(K_0)$ term.

 \end{proof}

\subsection{Two sided cuspidal geodesics}
For the cases of interest here, the lattice $\G$ has a cusps at $\infty$ and at $0$ and it is natural to consider shears of the two sided cuspidal geodesic connecting them,
that is,
$$\mu_T(\Psi)=\int_0^\infty \Psi(Tx+i y)\frac{dy}{y}.$$
It is easy to see that $\mu_T$ is invariant under scaling, $\mu_T(\Psi)=\mu_T(\Psi^{a_t})$ and a simple computation shows that under 
$\sigma=\left(\begin{smallmatrix} 0 & 1 \\ -1 & 0\end{smallmatrix}\right)$ it transforms via $\mu_T(\Psi^\sigma)=\mu_{-T}(\Psi)$. In fact, using that $\Psi^\sigma(z)=\Psi(\tfrac{-1}{z})$ and making a change of variables gives the identity.
\begin{equation}\label{e:invert}
\mu_T(\Psi)=\int_{\tfrac{1}{\sqrt{T^2+1}}}^\infty\Psi(Ty+iy)\frac{dy}{y}+\int_{\tfrac{1}{\sqrt{T^2+1}}}^\infty \Psi^\sigma(-Ty
+i y)\frac{dy}{y}
\end{equation}
Applying our results on shears of cuspidal geodesic rays we get the following. 
\begin{cor}\label{c:twosided}
Let $\G$ be conjugate to $\G_0(p)$ with cusps at $\infty$ and $0$. Let $\omega_1,\omega_2$ denote the widths of the cusps of $\G$ and $\G^\sigma$ at $\infty$ respectively and let $\omega=\sqrt{\omega_1\omega_2}$.
For any $\alpha>\frac{7+12\theta}{3+6\theta},\;\eta_1<\frac{1}{2+4\theta}$ and $T\geq \omega^{1/\eta_1}$ we have 
for any $\Psi\in C^\infty_c(\G\bk \bH)$, 
\begin{eqnarray*}
\mu_{T}(\Psi)&=&2\mu_{\G}(\Psi)\log(T\omega)+\langle \cK_{\G,\infty},\Psi\rangle+\langle \cK_{\G,0},\Psi \rangle\\
&&+O_{\alpha,\eta_1}\left(\frac{\cS_\alpha(\Psi)\omega^{\frac12}}{T^{\frac{\eta_1}{2}}}+\frac{\cS_{\alpha_2}(\Psi)\omega^{\frac{1}{4}}}{T^{\frac{\eta_1}{2}+\frac{1}{8}}}\right),
\end{eqnarray*}
where $\cS_{\alpha},\cS_{\alpha+1}$ are as in \thmref{t:Equidistribution1}. \end{cor}
\begin{proof}
Note that if $\Psi(z)$ has period $\omega_1$ and $\Psi^\sigma$ has period $\omega_2$ then $\tilde\Psi(z)=\Psi(\sqrt{\tfrac{\omega_2}{\omega_1}}z)$ satisfies that both $\tilde\Psi$ and $\tilde\Psi^\sigma$ have a period of $\omega=\sqrt{\omega_1\omega_2}$. Since $\mu_T(\Psi)=\mu_T(\tilde\Psi)$, using \thmref{t:Equidistribution1} with $\tilde\Psi$ and $\tilde\Psi^\sigma$ in each part of  \eqref{e:invert} gives the claimed result.
\end{proof}

\section{Lattice points in cones}\label{sec:5}
We now apply our results on equidistribution of shears to get effective counting estimates for counting lattice points in cones of the form
\begin{equation}\label{e:CT}
\cC_T=\{z\in \bH: |\Re(z)|\leq T\Im(z)\}
,
\qquad\qquad
T\to\infty.
\end{equation}
For $\G$ conjugate to $\G_0(p)$ and $\tau\in\SL_2(\Q)$, we define the counting function
\begin{equation}\label{e:ConeCount}
\cN_{\cC_T}^\tau(\G)=\#\{\g\in \G: \tau^{-1}\g i\in \cC_T\}.
\end{equation}
\begin{thm}\label{p:ConeCount}
For any $\eta<\frac{3}{40+72\theta}$ and $\beta_1>1+2\theta$, we have for $T>(\gw_\tau \gw_{\tau\gs})^\gb$ that
$$\cN_{\cC_T}^{\tau}(\G)=\frac{2T}{v_\G}\left(\log(T^2\gw_\tau \gw_{\tau\sigma})-2+v_\G  (\cK_{\G,\fa_\tau}(i)+\cK_{\G,\fb_\tau}(i))+O\big(\frac{(\gw_\tau \gw_{\tau\sigma})^{\beta_1\eta}}{T^\eta}\big)\right).$$
Here $\gw_\tau$ denotes the width of the cusp at $\infty$ of $\G^\tau=\tau\G\tau^{-1}$, $\fa_\tau=\tau \infty$, and $\fb_\tau=\tau0$. 
\end{thm}
As a first step we show that the cones $\cC_T$ are well rounded. 

\begin{lem}\label{l:round}
Let $g\in G$ such that $g i\in \cC_T$ with $T\geq 1$. Then for any $h\in B_\delta$ we have $gh i\in \cC_{T(1+4\delta)}$. 
\end{lem}
\begin{proof}
For any $h\in B_\delta$ we have that $|\Re(h i)|\leq \delta$ and $|\Im(h i-1)|\leq 2\delta$ Write $g i=x+iy$ then we can write 
$gh i=x+y(\xi+i\eta)=x'+iy'$ with $|\xi|\leq \delta$ and $|\eta-1|\leq 2\delta$. We can thus write
$$\frac{|x'|}{y'}=\frac{|x+\xi y|}{\eta y}\leq \frac{|x|}{y}(\frac{1+|\xi y/x|}{\eta}).$$

We now consider two cases, first assume that $\frac{|x|}{y}\geq 1$, in which case 
$$\frac{|x'|}{y'} \leq T(\frac{1+\delta}{1-2\delta})\leq T(1+4\delta).$$
Next, when $|x|\leq y$ we bound
$$\frac{|x'|}{y'}=\frac{|x+\xi y|}{\eta y}\leq \frac{y(1+|\xi|)|}{\eta y}\leq \frac{1+\delta}{1-2\delta}\leq 1+4\delta.$$
\end{proof}

Now let $\chi_{\cC_T}$ denote the indicator function of $\cC_T$ and consider the function 
\begin{equation}\label{e:FT}
F_{T,\tau}(g)=\sum_{\g\in \G} \chi_{\cC_T}(\tau^{-1} \g g.i)
\end{equation}
Note that $F_{T,\tau}\in L^2(\G\bk G/K)$ and that evaluating at the identity we have $F_{T,\tau}(1)=\cN_{\cC_T}^\tau(\G)$.
From the well roundedness of $\cC_T$ we get the following:

\begin{lem}\label{l:Count2Inner}
For $\delta>0$ small, let $\psi_\delta\in C^\infty_c(K\bk G/K)$ be supported on $B_\delta$ with $\int_G \psi_\delta=1$, and let 
\be\label{eq:PsiDelDef}
\Psi_\delta(g)=\sum_{\g\in \G} \psi_\gd(\g g).
\ee
Then $\Psi_\delta\in L^2(\G\bk G/K)$ and 
\begin{equation}\label{e:unitaprox}
\langle F_{T(1-4\delta),\tau},\Psi_\delta\rangle \leq \cN_{\cC_T}^\tau(\G)\leq \langle F_{T(1+4\delta),\tau},\Psi_\delta \rangle,
\end{equation}
where the inner product the standard inner product in $L^2(\G\bk G/K)$.
\end{lem}
\begin{proof}
Unfolding $\Psi_\delta$ we can write
\begin{eqnarray*}
\langle F_{T,\tau},\Psi_\delta\rangle&=& \int_G F_{T,\tau}(g)\psi_\delta(g)dg.
\end{eqnarray*}
By \lemref{l:round} we have that for any $g\in B_\delta$
$$F_{T(1-4\delta),\tau}(g)\leq \cN_{\cC_T}^\tau(\G)\leq F_{T(1+4\delta),\tau}(g),$$
concluding the proof.
\end{proof}

\begin{rmk}
As is well known, this inner product $\<F_T,\Psi_\gd\>$ not only is an approximation to the sharp cutoff $\cN_{\cC_T}(\G)$, but it is also itself a smooth counting function, since
$$
\<F_T,\Psi_\gd\>
=
\sum_{\g\in\G}
\widetilde{\chi_{\cC_T}}(g),
$$
where
$$
\widetilde{\chi_{\cC_T}}(g)
=
\int_{g\in G}\chi_{\cC_T}(\g g)\psi_\gd(g)dg,
$$
is a smoothed cutoff, cf. \eqref{eq:psiTilDef}.
\end{rmk}

We thus need to evaluate the inner product  $\langle F_{T,\tau},\Psi_\delta \rangle$. The following simple  lemma, relates these to the shears of two sided cuspidal geodesics.
\begin{lem}\label{l:Inner2shear}
For any $\Psi\in C^\infty_c(\G\bk \bH)$ we have 
\begin{equation}
\langle F_{T,\tau},\Psi\rangle= \int_{-T}^{T} \int_0^\infty \Psi^\tau(xy+iy)\tfrac{dy}{y}dx
\end{equation}
\end{lem}
\begin{proof}
Unfolding $F_{T,\tau}$ and making some changes of variables gives 
\begin{eqnarray*}
\langle F_{T,\tau},\Psi\rangle&=&\int_G \chi_{\cC_T}(\tau^{-1} gi)\Psi(gi)dg\\
&=&\int_G \chi_{\cC_T}( gi) \Psi^\tau (gi)dg\\
&=&\int_0^\infty \int_{-yT}^{yT}  \Psi^\tau(x+iy)\tfrac{dxdy}{y^2}\\
&=&\int_0^\infty \int_{-T}^{T}  \Psi^\tau (xy+iy)\tfrac{dxdy}{y}\\
&=& \int_{-T}^{T} \int_0^\infty \Psi^\tau(xy+iy)\tfrac{dy}{y}dx.\\
\end{eqnarray*}
as claimed.
\end{proof}

We are in position now to prove \thmref{thm:2}, which follows easily from the following:
\begin{prop}\label{prop:smoothC}
Let $\gd>0$ be sufficiently small and $\Psi=\Psi_\gd$ as in \eqref{eq:PsiDelDef}. 
Assume that $\alpha>\frac{7+12\theta}{3+6\theta}$ and $\eta_1<\frac{1}{2+4\theta}$.
Then
$$
\<F_{T,\tau},\Psi\>=
{2T\over v_\G}
\bigg(
\log( T^2 \gw_\tau\gw_{\tau\gs})
-2
+v_\G
(\<K_{\G^\tau,\infty},\Psi^\tau\>+\<K_{\G^\tau,0},\Psi^\tau\>)
\bigg)
$$
$$
\hskip1in
+
O_{\ga,\eta_1}
\bigg(
\cS_\alpha(\Psi)(\gw_\tau \gw_{\tau\sigma})^{\frac{1}{4}}T^{1-\frac{\eta_1}{2}}+\cS_{\alpha_2}(\Psi)(\gw_\tau \gw_{\tau\sigma})^{\frac{1}{8}}T^{7/8-\frac{\eta_1}{2}}\bigg)
$$
\be\label{eq:Prop512}
\hskip1in
+
O_{\eta_1}\bigg(
(\omega_\tau\omega_{\tau\sigma})^{1/2\eta_1}
\log(\omega_\tau\omega_{\tau\sigma})
\|\Psi\|_\infty
\bigg)
.
\ee

\end{prop}

\begin{proof}
%

By  \lemref{l:Inner2shear}, we have 
\begin{equation}\label{e:inner}
\langle F_{T,\tau},\Psi\rangle= \int_{-T}^{T} \int_{0}^\infty \Psi^{\tau}(xy+iy)\tfrac{dy}{y}dx
.
\end{equation}

Let 
$$
M=(\omega_\tau\omega_{\tau\sigma})^{1/2\eta_1}
.
$$
Then for $|x|>M$,
we may apply  \corref{c:twosided} 
to 
 the inner integral to obtain
\beann
\int_{0}^\infty \Psi^\tau(xy+iy)\tfrac{dy}{y}
&=&
\frac{\log(|x|^2\gw_\tau \gw_{\tau\sigma})}{v_\G}+\langle \cK_{\G^\tau,\infty},\Psi^\tau\rangle+\langle \cK_{\G^\tau,0},\Psi^\tau\rangle
\\
&&
+O_{\alpha,\eta}(\frac{\cS_\alpha(\Psi)(\gw_\tau \gw_{\tau\sigma})^{\frac{1}{4}}}{|x|^{\frac{\eta_1}{2}}}+\frac{\cS_{\alpha_2}(\Psi)(\gw_\tau \gw_{\tau\sigma})^{\frac{1}{8}}}{|x|^{1/8+\frac{\eta_1}{2}}})
.
\eeann
Integrating over $M\leq |x|\leq T$ gives the first three terms in \eqref{eq:Prop512}.

For $|x|<M$, we argue as follows. 
First fix $x$ and apply \eqref{e:invert} to the  $y$ integral: 
$$
\int_{0}^\infty \Psi^\tau(xy+iy)\tfrac{dy}{y}
=
\int_{1\over \sqrt{1+x^2}}^\infty \Psi^\tau(xy+iy)\tfrac{dy}{y}
+
sim
,
$$
where ``sim'' refers to a similar term with $\tau$ replaced by $\tau\gs$ and $x$ by $-x$.
Break the $y$ integral into 
$$
\int_{1\over \sqrt{1+x^2}}^\infty=\int_{1/\sqrt{x^2+1}}^{1/\gw_\tau}+\int_{1/\gw_\tau}^\infty.
$$
(If $1/\gw_\tau<1/\sqrt{x^2+1}$, then we only keep the second term.)
We claim that the curve in this second term passes through only one ``standard'' fundamental domain for $\G^\tau$ (that is, a Dirichlet domain centered at a point in the injective horoball centered at $\infty$). Indeed,
recall that $\G$ is conjugate to $\G_0(p)$, which is a subgroup of $\PSL_2(\Z)$. But the length 1 horocycle centered at $\infty$ is injective for the latter, and  hence also for the former. The width of the cusp at $\infty$ of $\G^\tau$ is $\gw_\tau$, and hence this  length 1 horocycle is at (Euclidean) height $y=1/\gw_\tau$.

On this one fundamental domain, 
 $\Psi^\tau$ is supported on a small $\delta$-neighborhood of $\tau^{-1}i$, which occurs in some range, say, 
 $$
 y\in(y_0,y_0(1+C\gd)).
 $$ 
 Hence
 $$
 \int_{1/\gw_\tau}^\infty
  \Psi^\tau(xy+iy)\tfrac{dy}{y}
 \ \ll\
\|\Psi\|_\infty
 \int_{y_0}^{y_0(1+C\gd)}
\frac{dy}{y}
\  \ll\
\gd
\|\Psi\|_\infty
.
 $$
 
 For the other interval, we trivially bound
$$
\int_{1/\sqrt{x^2+1}}^{1/\gw_\tau}  \Psi^\tau(xy+iy)\tfrac{dy}{y}
\ \ll\
\|\Psi\|_\infty
\log{(3+x)\over \gw_\tau}
.
$$

Finally, we integrate over $|x|<M$; since $\gd<1$, the term $\gd M\|\Psi\|_\infty$ is not dominant and may be dropped. Repeating the argument with $\tau$ replaced by $\tau\gs$ gives
 the final term of  \eqref{eq:Prop512}.
\end{proof}

We are finally in position to give the following
\begin{proof}[Proof of \thmref{p:ConeCount}]
Let $\ga$ and $\eta_1$ be as in \propref{prop:smoothC}. Let $\eta=\frac{\eta_1}{2(\alpha+1)}$ and $\beta_1=\frac{1}{2\eta_1}$ and denote by $\omega=\sqrt{\omega_\tau\omega_{\tau\sigma}}$.
We may choose our $\delta$-bump function $\psi_\delta$ so that 
$$\cS_\alpha(\Psi_\delta)\asymp \delta^{-\alpha},\quad \cS_{\alpha_2}(\Psi_\delta)\asymp \delta^{-\alpha_2},\quad \|\Psi\|_\infty\asymp \delta^{-2} $$ 
where $\cS_\alpha,\cS_{\alpha_2}$ are as in  \thmref{t:Equidistribution1}.

Now we further simplify the  Eisenstein terms appearing in \propref{prop:smoothC}. Since $E_{\G^\tau,\infty}(z,s)=E^\tau_{\G,\fa_\tau}(z,s)$ removing the residue we get that 
$\cK_{\G^\tau,\infty}(z)= \cK^\tau_{\G,\fa_\tau}(z)$ and  hence, $\langle \cK_{\G^\tau,\infty},\Psi_\delta^\tau\rangle_{\G^\tau}=\langle \cK_{\G,\fa_\tau},\Psi_\delta \rangle_\G$.
Since $\cK_{\G,\fa_\tau}(z)$ is smooth and $\Psi_\delta$ is supported on a $\delta$-neighborhood of $i$  we can estimate
$\langle \cK_{\G,\fa_\tau},\Psi_\delta \rangle_\G=\cK_{\G,\fa_\tau}(i)+O(\delta)$ so the Eisenstein term can be approximated by
\begin{eqnarray*}
\langle \cK_{\G^\tau,\infty},\Psi_\delta^\tau\rangle_{\G^\tau}= \cK_{\G,\fa_\tau}(i)+O(\delta)
.
\end{eqnarray*}

Making an optimal choice of 
\begin{equation}\label{e:delta}
\delta=\omega^{\frac{1}{2(1+\alpha)}}T^{-\frac{\eta_1}{2(1+\alpha)}},\end{equation}
the first error term dominates and we get 
\begin{eqnarray*}\label{e:innerapprox}
\langle F_{T,\tau},\Psi_\delta\rangle&=&\frac{2T}{v_\G}\left(\log(T^2\gw^2)-2+v_\G  (\cK_{\G,\fa_\tau}(i)+\cK_{\G,\fa_{\tau\sigma}}(i))\right)\\
\nonumber &&+O\big(\omega^{2\beta_1\eta }T^{1-\eta}\big).
\end{eqnarray*}

Finally, using \eqref{e:unitaprox} relating the counting problem to the inner product 
(after replacing $T$ with $T(1\pm 4\delta)$) we get that
$$\cN_{\cC_T}^{\tau}(\G)=\frac{2T}{v_\G}\left(\log(T^2\gw)-2+v_\G  (\cK_{\G,\fa_\tau}(i)+\cK_{\G,\fb_\tau}(i))+O((\tfrac{\omega^{2\beta_1}}{T})^\eta\right).$$
\end{proof}
\begin{rem}
As can be expected, the main term does not depend on $\tau$. The secondary term does depend on $\tau$, but only involves
knowledge of which cusps are used for the Eisenstein term, and
 the widths of these cusps.
 \end{rem}


\section{Counting integer solutions}\label{sec:6}
In this section, we establish the results claimed in \secref{sec:1.3}, handling \thmsref{thm:3} and \ref{thm:4} simultaneously.
\subsection{Decomposition into orbits}\label{sec:spin}
Consider the variety  
$$V_d\ : \ 
 b^2-4ac=d,$$ 
 where $d$ is fixed.
Identifying the triple $(a,b,c)$ with the quadratic form 
$$Q(x,y)=ax^2+bxy+cy^2,$$ 
gives a natural $\SL_2$ action on $V_d$, via
$$Q^g(v)=Q(vg^t),$$ 
where $g\in \SL_2$ acts linearly on $v=(x,y)$ from the left (here $g^t$ is the transpose of $g$).
More explicitly, the action of $g\in \SL_2$ on triples $(a,b,c)\in V_d$  is given by the linear action
$$(a,b,c)^g:=(a,b,c)\iota(g).$$
where $\iota:\SL_2\to \SO_{b^2-4ac}$ is the spin morphism given by
$$\iota\left(\begin{smallmatrix}a& b\\ c& d\end{smallmatrix}\right)=\begin{pmatrix} a^2 & 2ab &b^2\\ ac & ad+bc & bd\\ c^2 & 2cd & d^2\end{pmatrix}.$$

In particular, $\G_1=\SL_2(\Z)$ acts on the integer points $V_d(\Z)$ and we can decompose
$$V_d(\Z)=\cup_{i=1}^{h(d)}v_i \G_1,$$
into finitely many orbits.
The number $h(d)$ of orbits is, in general, 
 very mysterious;
 for instance, when $d$ is a square free fundamental discriminant then $h(d)$ is the class number of the quadratic extension $\Q(\sqrt{d})$. However when $d=n^2$ is a perfect square, the number of orbits can be computed explicitly. In the following lemma, we compute it and give a full set of representatives for the orbits. 

\begin{lem}\label{l:orbits}
For $d=n^2$ a square we have $h(n^2)=n$. Moreover,
the set $\{(0,n,0)^{\tau_j}|0\leq j<n\}$ with $\tau_j=\left(\begin{smallmatrix} 1 & j/n\\ 0 & 1\end{smallmatrix}\right)$ 
is a full set of representatives for the classes of  $V_{d}(\Z)/\SL_2(\Z)$, 
\end{lem}
\begin{proof}
We identify a point $(a,b,c)\in V_d(\Z)$ with the corresponding quadratic form $Q(x,y)=ax^2+bxy+cy^2$ having discriminant $d=b^2-4ac$. Then $SL_2(\Z)$ acts on the set of quadratic forms by $Q^\gamma(v)=Q(v\gamma^t)$ where $v=(x,y)$ and $\gamma\in \SL_2(\Z)$ is acting linearly on the left.

Recall that a binary quadratic form $Q$ has a square discriminant, if and only if the form factors as a product of linear forms 
$$Q(x,y)=(A x+B y)(C x+Dy),$$ 
in which case the discriminant is given by $(AC-BD)^2$. We thus get a map, from the set 
$\cM_n$ of $2\times 2$ matrices with determinant $n$ onto the set of quadratic forms of discriminant $n^2$, sending $M=\begin{pmatrix}A & B\\ C& D\end{pmatrix}$ to the form $Q_M(x,y)=(Ax+By)(Cx+Dy)$ (this map is not injective since the same form can have several factorizations). A direct computation shows that $Q_M^\gamma=Q_{M\gamma}$, and since 
$$\Delta_n=\{\begin{pmatrix} A & 0\\ C & D\end{pmatrix}: AD=n,\; 0\leq C<D\},$$ is a full set of representatives for $\cM_n/\G$, the set $\{Q_M, M\in \Delta_n\}$ is a full set of representatives for classes of quadratic form of discriminant $n^2$ (some of these might be equivalent though).
A form in this set of representatives can be written explicitly as
$$Q_M(x,y)=Ax(Cx+Dy)=ACx^2+nxy,$$ and we see that the classes with $A\neq 1$ are redundant. A full set of inequivalent representatives is thus given by the forms $Cx^2+nxy$ with $0\leq C<n$.  Finally, observing that $Cx^2+nxy$ is equivalent to $nxy-Cy^2$ concludes the proof.
\end{proof}
\begin{rem}
Instead of looking at all integer points in $V_d(\Z)$ one can consider only primitive points (i.e., points with gcd  $(a,b,c)=1$). It is easy to see that $\G_1$ also acts on the set of primitive points, and the same proof shows that the set $\{(0,n,j)|0\leq c<n,\; (n,j)=1\}$ is a full set of representatives for the orbits of primitive points. 
\end{rem}

Using our orbit decomposition of $V_d$ we can also get a corresponding decomposition of the variety 
$$W_d \ : \ 
x^2+y^2-z^2=d 
.$$
The map  $(x,y,z)\mapsto (\frac{z+y}{2},x,\tfrac{z-y}{2})$ is a bijection between $W_d$ and $V_d$ and the integer points $W_d(\Z)$ map to the set 
$$\widetilde V_d(\Z):=\{(a,b,c)\in V_d| b\in \Z,\; a,c\in \tfrac12\Z,\; a+c\in \Z\}.$$
From this map we see that the congruence subgroup 
$$\G_2=\{\g\in \G: \bar\gamma\in\{ \left(\begin{smallmatrix} 1& 0\\ 0 &1\end{smallmatrix}\right), \;\left(\begin{smallmatrix} 0& 1& \\-1 &0\end{smallmatrix}\right)\}\},$$
(with $\bar\g\in \SL_2(\Z/2\Z)$ the projection of $\g$),
acts on $\widetilde V_d(\Z)$ (and hence also on $W_d(\Z)$). Using the classification of the orbit of the $\G_1$ action on $V_d(\Z)$ we get the following classification for the $\G_2$ action on $\widetilde V_d(\Z)$.
\begin{lem}
For $d=n^2$ a complete set of representatives for the $\G_2$ orbits of $\widetilde V_d(\Z)$ are given by 
$$\{(0,n,0)^{\tau_j},\; 0\leq j<2n\}\cup \{(0,n,0)^{\widetilde\tau_j}: 0\leq j<2n\},$$
where $\tau_j=\left(\begin{smallmatrix} 1 & j/n\\ 0 & 1\end{smallmatrix}\right)$ is as above and $\widetilde\tau_j=\left(\begin{smallmatrix} 1+\frac{j}{2n}& \frac{j}{2n}& \\1 &1\end{smallmatrix}\right)$.
\end{lem}
\begin{proof}
Let $(a,b,c)\in \widetilde V_d(\Z)$ then $(2a,2b,2c)\in V_{4d}(\Z)$. From the classification of $\G_1$ orbits of $V_{4d}(\Z)$ there is $\g\in\G_1$ with $(a,b,c)=(0,n,j/2)^\g$ with $0\leq j<2n$. We can write $\g$ as $\sigma\widetilde\g$ with $\widetilde\g\in \G_2$ and $\sigma\in \{ \left(\begin{smallmatrix} 1& 0\\ 0 &1\end{smallmatrix}\right), \left(\begin{smallmatrix} 1& 1\\ 0 &1\end{smallmatrix}\right), \left(\begin{smallmatrix} 1& 0\\ 1 &1\end{smallmatrix}\right)\}$ in the set of representatives for $\G_1/\G_2$. We can thus write the point $(a,b,c)$ as
$(0,n,j/2)^{\widetilde\g}$ or as $(0,n,n+j/2)^{\widetilde\g}$ or $(n+\tfrac{j}{2},n+j,\tfrac{j}{2})^{\widetilde\g}$ for some $0\leq j<2n$. Now, $(a,b,c)\in \widetilde V_d(\Z)$ and  $\widetilde{\g}$ preserves this space, so in the first two cases we must have that $j/2\in \Z$. Hence, a full set of representatives are indeed $(0,n,j)=(0,n,0)^{\tau_j}$ and $(n+\tfrac{j}{2},n+j,\tfrac{j}{2})=(0,n,0)^{\sigma_j}$ for $0\leq j<2n$.
\end{proof}
\begin{rem}
Let $B_T=\{(a,b,c)\in \R^3: 2a^2+b^2+2c^2<T_2\}$ and note that $\widetilde V_d(\Z)\cap B_T$ is in bijection with the set $\{(x,y,z)\in W_d(\Z): x^2+y^2+z^2\leq T^2\}$ so 
$\cN_d(T)=\#(\widetilde V_d(\Z)\cap B_T)$. 
Moreover, under this bijection, the primitive points of $W_d(\Z)$ correspond exactly to the $\G_2$-orbits of the classes $(0,n,0)^{\tau_j}$ and  $(0,n,0)^{\widetilde\tau_j}$ with $(j,n)=1$.
\end{rem}

Decomposing the integer points of $V_d(\Z)\cap B_T$ and $\widetilde V_d(\Z)\cap B_T$ into the finitely many orbits, it is enough to count points in each orbit separately. For this we need to
estimate terms of the form
\begin{eqnarray*}
\#\{\g\in \G: \| (0,n,0)^{\tau\g}\|\leq T\}
\end{eqnarray*}
with $\tau=\tau_j$ or $\tau=\widetilde\tau_j$ as above and the lattice  $\G=\G_1$ or $\G=\G_2$ respectively.  
We now show that these counting functions are  given in therms of the cone counting function defined in \eqref{e:ConeCount}.
\begin{lem}\label{l:cone}
For any lattice $\G$ and $\tau\in G$ we have
$$\#\{\g\in \G: \|(0,n,0)^{\tau\g}\|\leq T\}=\cN_{\cC_{T_n}}^{\tau^{-1}}(\G)$$
with $T_n=\sqrt{\frac{T^2}{2n^2}-\frac{1}{2}}$.
\end{lem} 
\begin{proof}
Write $\tau\gamma=a_yn_xk$ so that $\tau \g i=yx+iy$.
Since $ (0,n,0)^{a_y}=(0,n,0)$ and our norm is $K$-invariant we can explicitly compute 
$$\| (0,n,0)^{\tau\g}\|^2=n^2(1+2x^2),$$ 
so that indeed $\| (0,n,0)^{\tau\g}\|\leq T$ if and only if $|x|\leq T_n$ which is equivalent to $\tau\g.i\in \cC_{T_n}$.
\end{proof}

\subsection{Square discriminants}
Our goal here is to prove \thmref{thm:4} by estimating
\begin{eqnarray*}
\#V_d(\Z)\cap B_T&=&\sum_{j=0}^{n-1}\#\{\g\in \G_1: \| (0,n,0)^{\tau_j\g}\|\leq T\}\\
&=&\sum_{j=0}^{n-1}\cN_{\cC_{T_n}}^{\tau^{-1}_{j}}(\G_1)
.
\end{eqnarray*}
Note that $\G_1^{\tau^{-1}_j}$ has a cusp at $\infty$ of width $1$ and $\G_1^{\tau^{-1}_j\sigma}$ has a cusp at $\infty$ of width $\omega_j=\frac{n^2}{(n,j)^2}$.
%
Hence, from  \thmref{p:ConeCount} we get that
$$\cN_{\cC_{T_n}}^{\tau_j^{-1}}(\G_1)=\frac{2T_n}{v_{\G_1}}\left( 2\log(T_n)+\log(\frac{n^2}{(n,j)^2})+2v_\G  \cK_{\G_1}(i)+O(\frac{n^{2\beta_1\eta}}{T_n^{\eta}})\right),$$
with $\beta_1=\beta-\tfrac{1}{2}>1+2\theta$. Recalling the assumption $T\geq d=n^2$ in the statement of \thmref{thm:1}, we have that $T_n>1$ and we can
estimate $T_n=\frac{T}{\sqrt{2}n}+O(\frac{n}{T})$, and 
$$\log(T_n)=\log(T)-\tfrac{1}{2}\log(2)-\log(n)+O(\frac{n^2}{T^2}),$$ 
so that 
$$\cN_{\cC_{T_n}}^{\tau_j^{-1}}(\G_1)=\frac{\sqrt{2}T}{nv_{\G_1}}\left(2\log(T)-\log(2)-2\log((n,j))+2v_\G  \cK_{\G_1}(i)+O(\tfrac{n^{(2\beta_1+1)\eta}}{T^\eta})\right).$$
Summing over all orbits we get that 
$$\#V_d(\Z)\cap B_T=\frac{\sqrt{72}T}{\pi}\left(\log(T)-\tfrac12\log(2)+\tfrac\pi 3  \cK_{\G_1}(i)-\frac 1n \sum_{j=1}^n\log(n,j)+O(\tfrac{n^{2\beta\eta}}{T^\eta})\right).$$
Plugging in the value of  $\cK_{\G_1}(i)$ from \eqref{eq:EisIs1} and noting that 
$$\sum_{j=1}^{n-1}\log(j,n)=\sum_{a|n}\phi(\tfrac{n}{a})\log(a),$$
concludes the proof of \eqref{e:NQ2}.

\subsection{Sum of squares}
Next we prove \thmref{thm:3} by estimating $\cN_d(T)=\widetilde V_d(\Z)\cap B_T$. Again, split the integral points into the finitely many $\G_2$ orbit and count in each orbit.
We thus need to estimate $\cN_{\cC_{T_n}}^{\tau_j^{-1}}(\G_2)$ and $\cN_{\cC_{T_n}}^{{\widetilde \tau}_j^{-1}}(\G_2)$.

Let  $\fa_j=\tau_j^{-1}\infty,\fb_j=\tau_j^{-1}0$ and let  $\omega_j,\omega_j'$ denote the width of the cusps at $\infty$ of $\G_2^{\tau_j^{-1}}$ and $\G_2^{\tau_j^{-1}\sigma}$, appearing in the formula for $\cN_{\cC_{T_n}}^{\tau_j^{-1}}(\G_2)$. Similarly let  $\widetilde\fa_j=
\widetilde \tau_j^{-1}\infty,\widetilde\fb_j=\widetilde\tau_j^{-1}0$ and let  $\widetilde \omega_j,\widetilde \omega_j'$ the corresponding cusp widths. 

Recall that  that $\G_2$ has only two inequivalent cusps, one at $\infty$ and another at $1$. After verifying which pair of cusps we get for each orbit, we show that the contribution of the Kronecker terms to the counting function is as follows.
\begin{lem}\label{l:Ksum}
With the above notation we have
$$\sum_{j=0}^{2n-1}(\cK_{\G_2,\fa_j}(i)+\cK_{\G_2,\fb_j}(i)+\cK_{\G_2,\widetilde\fa_j}(i)+\cK_{\G_2,\widetilde\fb_j}(i))=4n(\cK_{\G_2,\infty}(i)+\cK_{\G_2,1}(i)).$$
\end{lem}
\begin{proof}
Since $\tau_j^{-1}\infty=\infty$ and $\widetilde\tau_j^{-1}\infty=-1$ (which is $\G_2$ equivalent to the cusp at $1$)  we see that $\fa_j=\infty$ and $\widetilde\fa_j=1$. 
Next $\tau_j^{-1}0=\tfrac{j}{n}$ is $\G_1$-equivalent to the cusp at infinity by the action of some $\g=\left(\begin{smallmatrix} a &b\\c&d\end{smallmatrix}\right)\in \G_1$ such that $\tfrac{j}{n}=\g\infty=\tfrac{a}{b}$, that is, with $a=\tfrac{j}{(j,n)}$ and $b=\tfrac{n}{(j,n)}$. Such an element $\g$ lies in $\G_2$ iff $\tfrac{nj}{(j,n)^2}$ is even, so that 
$\fb_j=\left\lbrace\begin{array}{cc} \infty & \frac{nj}{(n,j)^2}=0\pmod{2}\\
1 & \frac{nj}{(n,j)^2}=1\pmod{2}\\
\end{array}\right.$. Similarly, $\widetilde\tau_j^{-1}0=\tfrac{-j}{2n+j}$ is $\G_1$-equivalent to the cusp at $\infty$ by $\g\in \G_1$ with $\tfrac{a}{b}=\tfrac{-j}{2n+j}$ so that $a=\tfrac{j}{(j,2n)}$ and $b=\tfrac{2n+j}{(j,2n)}$. Hence $\g$ can be taken from $\G_2$ iff $\tfrac{j(j+2n)}{(j,2n)^2}$ is even implying that  
$\widetilde\fb_j=\left\lbrace\begin{array}{cc} \infty & \frac{j(j+2n)}{(2n,j)^2}=0\pmod{2}\\
1 & \frac{(j+2n)j}{(2n,j)^2}=1\pmod{2}\\
\end{array}\right.$. We thus get that the term  $\cK_j=\cK_{\G_2,\fa_j}(i)+\cK_{\G_2,\fb_j}(i)+\cK_{\G_2,\widetilde\fa_j}(i)+\cK_{\G_2,\widetilde\fb_j}(i)$ is given by
$$\cK_j=
\left\lbrace\begin{array}{cc}
3\cK_{\G_2,\infty}(i)+\cK_{\G_2,1}(i)&  \frac{nj}{(n,j)^2}=\frac{j(j+2n)}{(2n,j)^2}=0\pmod{2} \\
\cK_{\G_2,\infty}(i)+3\cK_{\G_2,1}(i)& \frac{nj}{(n,j)^2}=\frac{j(j+2n)}{(2n,j)^2}=1\pmod{2}  \\
2\cK_{\G_2,\infty}(i)+2\cK_{\G_2,1}(i)& \frac{nj}{(n,j)^2}\neq \frac{j(j+2n)}{(2n,j)^2}\pmod{2}.
\end{array}\right.
$$
Writing $n=2^am$ with $m$ odd we see that the first case happens when $j=0\pmod{2^a}$ but $j\neq 0\pmod{2^{a+1}}$ (hence for $\frac{n}{2^a}$ values of $0\leq j<2n$), the second case when $j=0\pmod{2^{a+1}}$ (for another $\frac{n}{2^a}$ values of $j$) and the last when $j\neq 0\pmod{2^a}$ (for $2n(1-\frac{1}{2^a})$ values of $j$). Now summing over all $0\leq j<2n$ we get our result.
\end{proof}

Summing up the contributions from the widths of the cusps, we get
\begin{lem}\label{l:sumwj}
With notation as above, let $n=2^\nu m$ with $m$ odd. Then 
$$\sum_{j=0}^{2n-1}\log(\omega_j\omega_j'\widetilde {\omega_j}{\widetilde{\omega}}_j')=
8n\log(2n)-\tfrac{2n\log(2)}{2^\nu}-8\sum_{a|n}\phi(\tfrac{n}{a})\log(a).$$
\end{lem}
\begin{proof}
Fix $0\leq j<2n$. Since $\tau_j$ commutes with $N$ the width of the cusp at $\infty$ of $\G_2^{\tau_j^{-1}}$ is the same as for $\G_2$ so $\omega_j=2$. 
Similarly, $\widetilde\tau_j^{-1}\begin{pmatrix} 1 & 1\\ 0& 1\end{pmatrix}\widetilde\tau_j=\begin{pmatrix} 2 & 1 \\-1 & 0\end{pmatrix}\in \G_2$ 
so $\widetilde \omega_j=1$. 

Next $\omega_j'$ is the smallest integer $k$ such that 
$$\tau^{-1}_j\sigma \begin{pmatrix} 1 & k \\ 0&1\end{pmatrix}\sigma \tau_j=\begin{pmatrix} 1-\tfrac{kj}{n} & \tfrac{-j^2k}{n^2}\\ k& 1+\tfrac{kj}{n}\end{pmatrix}\in \G_2,$$
hence 
$$\omega_j'=\left\lbrace\begin{array}{cc}
\frac{n^2}{(n,j)^2}, & \tfrac{jn}{(j,n)^2}=1\pmod{2}\\
\frac{2n^2}{(n,j)^2}, & \tfrac{jn}{(j,n)^2}=0\pmod{2}.
\end{array}\right.$$
Similarly, ${\widetilde{\omega}}_j'$ is the smallest integer $k$ such that 
\begin{equation}\label{e:conjn_k}
{\widetilde\tau}^{-1}_j\sigma \begin{pmatrix} 1 & k \\ 0&1\end{pmatrix}\sigma \widetilde\tau_j=\begin{pmatrix} 1+\tfrac{kj}{2n}(1+\tfrac{j}{2n}) & \tfrac{j^2k}{4n^2}\\ k(1+\tfrac{j}{2n})^2& 1+\tfrac{kj}{2n}(1+\tfrac{j}{2n}).\end{pmatrix}\in \G_2,
\end{equation}
and a similar computation gives that 
$${\widetilde{\omega}}_j'=\left\lbrace\begin{array}{cc}
\frac{4n^2}{(2n,j)^2}, & \tfrac{j}{(j,2n)}=1\pmod{2} \mbox{ and } \tfrac{2n}{(j,2n)}=0\pmod{2} \\
\frac{8n^2}{(2n,j)^2},& \mbox{ otherwise.}\\
\end{array}\right.$$

Writing $n=2^\nu m$ with $m$ odd, it is not hard to see that
$$\omega_j{\widetilde{\omega}}_j\omega_j'{\widetilde{\omega}}_j'=\left\lbrace\begin{array}{cc}
\frac{8n^4}{(n,j)^4}, & j=0\pmod{2^\nu} \\
\frac{16n^4}{(n,j)^4},& \mbox{ otherwise,}
\end{array}\right.$$
so that 
\begin{eqnarray*}\sum_{j=0}^{2n-1}\log(\omega_j\omega_j'\widetilde {\omega_j}{\widetilde{\omega}}_j')
&=&\sum_{j=0}^{2n-1}\log(\frac{8n^4}{(n,j)^4})+2n\log(2)(1-\tfrac{1}{2^\nu})\\
&=&2n\log(8n^4)-8\sum_{j=0}^{n-1}\log((n,j))+2n\log(2)-\tfrac{2n\log(2)}{2^\nu}\\
&=&8n\log(2n)-\tfrac{2n\log(2)}{2^\nu}-8\sum_{a|n}\phi(\tfrac{n}{a})\log(a)
,
\end{eqnarray*}
as claimed.
\end{proof}

 \begin{proof}[Proof of \thmref{thm:3}]
 Partitioning $\widetilde V_d\cap B_T$ into $\G_2$ orbits and summing in each orbit gives 
 $$\cN_d(T)=\sum_{j=0}^{2n-1}\cN_{\cC_{T_n}}^{\tau_j^{-1}}(\G_2)+\cN_{\cC_{T_n}}^{{\widetilde \tau}_j^{-1}}(\G_2).$$
 Using  \thmref{p:ConeCount} and the estimates $T_n=\frac{T}{\sqrt{2}n}+O(\tfrac{n^2}{T^2})$ and $|\omega_j\omega_j'|\leq n^2$ we estimate each of the cone counting functions 
 $$\cN_{\cC_{T_n}}^{\tau_j^{-1}}(\G_2)=\frac{\sqrt{2}T}{\pi n}\left(\log(\tfrac{T^2\omega_j \omega_j'}{2n^2})-2+ \pi(\cK_{\G_2,\fa_j}(i)+\cK_{\G_2,\fb_j}(i))+O(\tfrac{n^{(2\beta_1+1)\eta}}{T^\eta})\right),$$
 with $\beta_1=\beta-\tfrac12$ as before, 
 and similarly for $\cN_{\cC_{T_n}}^{{\widetilde \tau}_j^{-1}}(\G_2)$. Summing over $0\leq j<2n$, by  \lemref{l:Ksum} the contribution of the Kronecker terms is 
 $\sqrt{32}T(\cK_{\G_2,\infty}(i)+\cK_{\G_2,1}(i))$ and by  \lemref{l:sumwj} the contribution of the terms $\log(\omega_j\omega_j'{\widetilde{\omega}}_j{\widetilde{\omega}}_j')$ is 
 \begin{eqnarray*}
 \sum_{j=0}^{2n-1}\log(\omega_j \omega_j' {\widetilde{\omega}}_j {\widetilde{\omega}}_j')=8n\log(2n)-\tfrac{2n\log(2)}{2^\nu}-8\sum_{a|n}\phi(\tfrac{n}{a})\log(a),
 \end{eqnarray*}
 so that 
\begin{eqnarray*}\cN_d(T)&=&\frac{\sqrt{32}T}{\pi }\bigg(2\log(T)-2+\log(2)+\pi(\cK_{\G_2,\infty}(i)+\cK_{\G_2,1}(i))\\
&&-\tfrac{\log(2)}{2^{\nu+1}}
-\frac{2}{n}\sum_{a|n}\phi(\tfrac{n}{a})\log(a)+
O_\eta(\tfrac{n^{2\beta}}{T^\eta})\bigg).
\end{eqnarray*}

We now want to express the Kronecker terms in terms of special values of Dedekind eta function. To do this, note that $\G_2=\G_0(2)^\tau$ with $\tau=\left(\begin{smallmatrix} 1 & 0 \\1& 1\end{smallmatrix}\right)$ and hence $E_{\G_2,\fa}(z,s)=E_{\G_0(2),\tau \fa}(\tau z,s)$ and also $\cK_{\G_2,\fa}(z)=\cK_{\G_0(2),\tau \fa}(z,s)$. In particular 
$$\cK_{\G_2,\infty}(i)+\cK_{\G_2,1}(i)=\cK_{\G_0(2),0}(\tfrac{i+1}{2})+\cK_{\G_0(2),\infty}(\tfrac{i+1}{2}).$$
Using  \propref{p:KroneckerLimit} for $\G=\G_0(2)$ we have 
$$\cK_{\G,\infty}(z)=\frac{1}{\pi}\left(2\g -2 \frac{\zeta'}{\zeta}(2)-\log\left(\tfrac{4y|\eta(2z)|^{8}}{|\eta(z)|^{4}}\right)- \frac{8\log(2)}{3} \right),$$
and
$$\cK_{\G,0}(z)=\frac{1}{\pi}\left(2\g -2 \frac{\zeta'}{\zeta}(2)-\log\left(\tfrac{4y|\eta(z)|^{8}}{|\eta(2z)|^{4}}\right)+ \frac{\log(2)}{3} \right).$$
Hence
$$\pi(\cK_{\G,0}(\tfrac{i+1}{2})+\cK_{\G,\infty}(\tfrac{i+1}{2}))=4\g-4 \frac{\zeta'}{\zeta}(2)-\tfrac{13\log(2)}{3}-2\log(|\eta(i+1)\eta(\tfrac{i+1}{2})|^2).$$
Using the transformation law for the Dedekind Eta function 
$$|\eta(z+1)|^2=|\eta(z)|^2,\quad |\eta(\tfrac{-1}{z})|^2=|z||\eta(z)|^2,$$
we have that $|\eta(i+1)\eta(\tfrac{i+1}{2})|^2=\sqrt{2}|\eta(i)|^4=\sqrt{2}\frac{\G(1/4)^4}{16\pi^{3}}$ so that 
$$\pi(\cK_{\G,0}(\tfrac{i+1}{2})+\cK_{\G,\infty}(\tfrac{i+1}{2}))=4\g-4 \frac{\zeta'}{\zeta}(2)-\tfrac{13\log(2)}{3}-2\log(2)-2\log(\frac{\G(1/4)^4}{16\pi^{3}}).$$
and plugging this back in we get that 
\begin{eqnarray*}\widetilde\cN_d(T)&=&\frac{\sqrt{128}T}{\pi }\bigg(\log(T)+C
-\frac{1}{n}\sum_{a|n}\phi(\tfrac{n}{a})\log(a)+\log(2)(\tfrac{1}{3}-\tfrac{1}{2^{\nu+2}})+
O(\tfrac{n^{2\beta\eta}}{T^\eta})\bigg)\end{eqnarray*}
with the constant 
$$C=2\g-1-2 \frac{\zeta'}{\zeta}(2)-\tfrac{5\log(2)}{2}-\log(\frac{\G(1/4)^4}{16\pi^{3}}),$$
as before.
This completes the proof.
 \end{proof}

 
\appendix
\section{Eisenstein Series for $\G_0(p)$}
As for the full modular group, $\G_1=\SL_2(\Z)$, the theory of Eisenstein series for the congruence groups $\G_0(p)$ is also well understood, in particular, the Fourier coefficients can be expressed explicitly and there is an explicit formula for the Kronicker limit. Since we could not found a suitable reference for these formulas we will include
 short proofs here, but we claim no originality. 

We first note that when $\G$ is a finite index subgroup of $\G_1$, one can express the Eisenstein series for $\G_1$ in terms of the Eisenstein series corresponding to the different cusps of $\G$. Explicitly, we have
\begin{lem}
Let $\sigma_1,\ldots, \sigma_k$ denote a complete set of representatives of $\G_1/\G$, and let $\fa_i=\sigma_j^{-1}\infty$ (these are not necessarily inequivalent cusps for $\G$). Then
\begin{equation}
E_{\G_1}(z,s)=\sum_{j=1}^k \omega_j^{s-1}E_{\G,\fa_j}(z,s)
\end{equation}
where $\omega_j$ denotes the width of the cusp $\fa_i$.
\end{lem}
\begin{proof}
Since $N\cap \sigma_j\G\sigma_j^{-1}\subseteq N\cap \sigma_j\G_1\sigma_j^{-1}=N\cap \G_1$ it is generated by $\left(\begin{smallmatrix} 1 & \omega_j \\ 0 & 1 \end{smallmatrix}\right)$ 
where the width $\omega_j\in \N$ of $\fa_j$ is the index of $\G_{\fa_j}$ in $\G_1\cap N$.
We can thus write for $\Re(s)>1$
\begin{eqnarray*}
E_{\G_1}(z,s)&=&
\sum_{(\G_1\cap N)\bk \G} \Im(\g z)^s\\
&=&\sum_{j=1}^k \omega_j^{-1}\sum_{\g\in {\G}_{\fa_i}\bk \G}\Im(\sigma_i\g z)^s
\end{eqnarray*}
Since $\tau_{\fa_i}=\sigma_i^{-1}a_{\omega_i}$ is a scaling matrix for $\fa_j$ we have
$$\sum_{\g\in \G_{\fa_i}\bk \G}\Im(\sigma_i\g z)^s=\sum_{\g\in \G_{\fa_i}\bk \G}\Im(a_{\omega_i}\tau_{\fa_i}^{-1}\g z)^s=\omega_i^sE_{\G,\fa_i}(z,s),$$
and the result follows.
\end{proof}
Subtracting the residue and taking the limit as $s\to 1$ we get the following
\begin{cor}
For $\G$ a subgroup of $\G_1$ the Kronecker limit satisfies
$$\cK_{\G_1}(z)=\sum_{j=1}^k \cK_{\G,\fa_i}(z)+\frac{3}{\pi k}\sum_{j=1}^k\log(\omega_j)$$
\end{cor}
In particular, applying this to the subgroup $\G=\G_0(p)$ of $\G_1=\SL_2(\Z)$ we get the following identities
\begin{equation}\label{e:EG2G0}
E_{\G_1,\infty}(z,s)=E_{\G_0(p),\infty}(z,s)+p^{s}E_{\G_0(p),0}(z,s),\end{equation}
\begin{equation}\label{e:KG2G0}
\cK_{\G_1}(z)=\cK_{\G_0(p),\infty}(z)+ p\cK_{\G_0(p),0}(z)+\frac{3p\log(p)}{(p+1)\pi}
\end{equation}

\subsection{Fourier Coefficients}
For each pair of cusps $\fa,\fb$ the Fourier expansion of the Eisenstein series $E_{\G,\fb}$ with respect to the cusp at $\fa$, is given by
$$E^{\tau_\fa}_{\G,\fb}(z,s)=\delta_{\fa,\fb}y^s+\phi_{\fa,\fb}(s)y^{1-s}+\sum_{m\neq 0}a_{\fa,\fb}(s;m,y)e(mx),$$
and since $E_{\G,\fb}(z,s)$ is an Eigenfunction with eigenvalue $s(1-s)$ we can write
$$a_{\fa,\fb}(s;m,y)=\phi_{\fa,\fb}(s;m)2\sqrt{y}K_{s-\frac12}(2\pi my).$$

For the full modular group $\G_1=\SL_2(\Z)$ there is just one cusp at $\infty$ and the Fourier coefficients are given explicitly by
$\phi(s)=\frac{\zeta^*(2s-1)}{\zeta^*(2s)}$ and
\begin{equation}
\phi(s,m)=\frac{\tau_{s-1/2}(m)}{\zeta^*(2s)},
\end{equation}
where $\zeta^*(s)=\pi^{-s/2}\zeta(s)\G(s/2)$ is the completed Riemann zeta function and  $\tau_s(m)=\sum_{ab=|m|}(\tfrac ab)^s$ is the divisor function 
\cite[page 67]{Iwaniec1995}.

For the congruence groups $\G_0(p)$ the Fourier coefficients can also given by a similar formula and satisfy a similar bound. 

\begin{prop}\label{p:phiEis}
For $\G=\G_0(p)$ we have 
$$\phi_{\infty,\infty}(s;m)=\phi_{0,0}(s;m)=\frac{1}{p^{2s}-1}
\left\lbrace\begin{array}{lc} -\phi(s;m) & (p,m)=1\\
p^{s+1/2}\phi(s;\tfrac{m}{p})-\phi(m;s) & p|m
\end{array}\right.
$$
$$\phi_{\infty,0}(s;m)=\phi_{0,\infty}(s;m)=\frac{1}{p^{2s}-1}
\left\lbrace\begin{array}{lc} -p^s\phi(s;m) & (p,m)=1\\
p^s\phi(s;m)- \sqrt{p}\phi(s;\tfrac{m}{p})\end{array}\right.
$$
\end{prop}
\begin{proof}
Since the scaling matrix $\tau_0$ normalizes $\G=\G_0(p)$ we have that 
$E_{\G,\infty}^{\tau_0}(z,s)=E_{\G,0}(z,s)$ implying that $\phi_{\infty,0}(s;m)=\phi_{0,\infty}(s;m)$, and since $\tau_0^2=1$ is the identity then 
$E_{\G,0}^{\tau_0}(z,s)=E_{\G,\infty}(z,s)$ implying that $\phi_{\infty,\infty}(s;m)=\phi_{0,0}(s;m)$.

Now looking at the expansion at infinity of  \eqref{e:EG2G0} we get that 
$$\phi(s;m)=\phi_{\infty,\infty}(s;m)+p^s\phi_{0,\infty}(s;m),$$
and the expansion at $0$ gives 
$$\sqrt{p}\phi(s;\tfrac{m}{p})=\phi_{\infty,0}(s;m)+p^s\phi_{0,0}(s;m),$$
where it is understood that $\phi(s;\tfrac{m}{p})=0$ when $(p,m)=1$. We thus get that
$$
\begin{pmatrix}
 1 & p^s\\
  p^s &1
\end{pmatrix}
\begin{pmatrix} \phi_{0,0}(s;m)\\ 
\phi_{0,\infty}(s;m)
\end{pmatrix}=
\begin{pmatrix}\phi(s;m)\\ \sqrt{p}\phi(s;\tfrac{m}{p})
\end{pmatrix}.
$$
and inverting the matrix concludes the proof.

\end{proof}

\subsection{Kronecker limits}
The Kronecker limit corresponding to a cusp $\fa$ is defined as the limit
\begin{equation}\label{e:tildeE}
\cK_{\G,\fa}(z)=\lim_{s\to 1}\left(E_{\G,\fa}(s,z)-\frac{1}{v_\G(s-1)}\right).
\end{equation}

When $\G_1=\SL_2(\Z)$ the Kronecker limit formula expresses $\cK_{\G_1}(z)$ explicitly in terms of the Dedekind $\eta$-function
 (see, e.g., \cite[(22.42), (22.63)--(22.69)]{IwaniecKowalski})
\be\label{eq:EisIs}
\cK_{\G_1}(z)
\ = \
{3\over \pi}
\left(
2\g
-2{\gz'\over\gz}(2)
-\log (4y|\eta(z)|^{4})
\right)
,\qquad
\ee
where $\g=0.577\cdots$ is Euler's constant, $\gz(s)$ is the Riemann zeta function, and $\eta(z)$ is the Dedekind eta function.
In particular, using the special value $\eta(i)^4=\frac{\G(1/4)^4}{16\pi^{3}}$ we see that at $z=i$ we have
\be\label{eq:EisIs1}
\cK_{\G_1}(i)
\ = \
{3\over \pi}
\left(
2\g
-2{\gz'\over\gz}(2)
-\log (\frac{\G(1/4)^{4}}{4\pi^3})
\right)
\ee

One can derive a similar formula for the congruence groups $\G=\G_0(p)$. 
\begin{prop}\label{p:KroneckerLimit}
For $\G=\G_0(p)$ we have 
$$\cK_{\G,\infty}(z)=\frac{3}{(p+1)\pi}\left(2\g -2 \frac{\zeta'}{\zeta}(2)-\log\left(\tfrac{4y|\eta(pz)|^{\frac{4p}{p-1}}}{|\eta(z)|^{\frac{4}{p-1}}}\right)- \frac{2\log(p)p^2}{(p^2-1)} \right),$$
and
$$\cK_{\G,0}(z)=\frac{3}{(p+1)\pi}\left(2\g -2 \frac{\zeta'}{\zeta}(2)-\log\left(\tfrac{4y|\eta(z)|^{\frac{4p}{p-1}}}{|\eta(pz)|^{\frac{4}{p-1}}}\right)+ \frac{\log(p)(p-1)^2}{(p^2-1)} \right).$$

\end{prop}
\begin{proof}
Since representatives for $\G_\infty\bk \G$ are given by matrices $\left(\begin{smallmatrix} * & *\\ c& d\end{smallmatrix}\right)$ with $c\geq 0$ integer with $c=0\mod{p}$ and $d\in \Z$ with $(c,d)=1$, after
multiplying the Eisenstein series by 
$$\zeta_p(2s)=\sum_{(n,p)=1}n^{-2s}=\zeta(2s)(1-p^{-2s}),$$ 
and expanding we get for $\Re(s)>1$
\begin{eqnarray*}
\zeta_p(2s)E_{\G,\infty}(z,s)&=&\zeta_p(2s)y^s+\sum_{(n,p)=1}\frac{1}{n^{2s}}\mathop{\sum_{(c,d)=1}}_{c=0\mod{p}}\frac{y^s}{|cz+d|^{2s}}\\
&=& \zeta_p(2s)y^2+\mathop{\sum_{c=1}^\infty}_{c=0\mod{p}}\sum_{(d,p)=1}\frac{y^s}{|cz+d|^{2s}}\\
&=& y^s\zeta_p(2s)+\mathop{\sum_{c=1}^\infty}_{c=0\mod{ p}}\sum_{d\in \Z}\frac{y^s}{|cz+d|^{2s}}- \frac{1}{p^{2s}}\sum_{c=1}^\infty\sum_{d\in \Z}\frac{y^s}{|cz+d|^{2s}}
.
\end{eqnarray*}
Using Poisson summation on the inner sum gives
$$\sum_{d\in \Z}\frac{1}{|cz+d|^{2s}}=\frac{\sqrt{\pi}\G(s-\tfrac12)}{\G(s)}(cy)^{1-2s}+(cy)^{1-2s}\sum_{m\neq 0}e(-mcx)\int_\R \frac{e(mcyt)}{(1+t^2)^s}dt$$
and dividing by $\zeta_p(2s)$ we see that 
\begin{eqnarray*}
E_{\G,\infty}(z,s)&=&
y^s+\tfrac{(p-1)\sqrt{\pi}\G(s-\tfrac12)\zeta(2s-1)}{(p^{2s}-1)\G(s)\zeta(2s)}y^{1-s}\\
&&+\frac{py^{1-s}}{(p^{2s}-1)\zeta(2s)}\sum_{c=1}^\infty c^{1-2s}\sum_{m\neq 0}e(-mpcx)\int_\R \frac{e(mpcyt)}{(1+t^2)^s}dt \\
&&-\frac{y^{1-s}}{(p^{2s}-1)\zeta(2s)}\sum_{c=1}^\infty c^{1-2s}\sum_{m\neq 0}e(-mcx)\int_\R \frac{e(mcyt)}{(1+t^2)^s}dt
.
\end{eqnarray*}
The only pole comes from the term containing $\zeta(2s-1)=\frac{1}{2(s-1)}+\g+O(s-1)$. Subtracting the residue 
$\Res_{s=1}E_{\G,\infty}(z,s)=\frac{3}{\pi(p+1)}$ and taking the limit as $s\to 1$ we get that 
\begin{eqnarray*}
\cK_{\G,\infty}(z)&=&
y+\lim_{s\to 1}\left(\frac{(p-1)\sqrt{\pi}\G(s-\tfrac12)\zeta(2s-1)}{(p^{2s}-1)\G(s)\zeta(2s)}y^{1-s}-\frac{3}{\pi(p+1)(s-1)}\right)\\
&&+\frac{p}{(p^2-1)\zeta(2)}\sum_{c=1}^\infty \frac{1}{c}\sum_{m\neq 0}e(-mpcx)\int_\R \frac{e(mpcyt)}{(1+t^2)}dt \\
&&-\frac{1}{(p^{2}-1)\zeta(2)}\sum_{c=1}^\infty \frac{1}{c}\sum_{m\neq 0}e(-mcx)\int_\R \frac{e(mcyt)}{(1+t^2)}dt
.
\end{eqnarray*}
We can evaluate the integral 
$$\int_\R \frac{e(at)}{(1+t^2)}dt=\pi e^{-2\pi |a|}
$$
to get that 
\begin{eqnarray*}
\sum_{c=1}^\infty \frac{1}{c}\sum_{m\neq 0}e(-mpcx)\int_\R \frac{e(mpcyt)}{(1+t^2)}dt &=&
\pi \sum_{c=1}^\infty \frac{1}{c}(\sum_{m=1}^\infty e^{2\pi i mpc z}+ e^{-2\pi  i mpc \bar z})\\
&=&\pi \sum_{m=1}^\infty \sum_{c=1}^\infty \frac{1}{c}(e^{2\pi i mpc z}+ e^{-2\pi  i mpc \bar z})\\
&=& -\pi  \sum_{m=1}^\infty(\log(1-e^{2\pi i mp z})+\log(1-e^{-2\pi  i mp \bar z}))
.
\end{eqnarray*}
Recalling the Dedekind $\eta$-function is given by
$$\eta(z)=e^{\pi i z/12}\prod_{m=1}^\infty(1-e^{2\pi i m z}),$$ 
we have that 
$\sum_{m=1}^\infty \log(1-e^{2\pi i m z})=\log(\eta(z))-\tfrac{\pi i z}{12}$  
hence 
\begin{eqnarray*}
\sum_{c=1}^\infty \frac{1}{c}\sum_{m\neq 0}e(-mpcx)\int_\R \frac{e(mpcyt)}{(1+t^2)}dt &=&
 -\pi\log(|\eta(pz)|^2)-\tfrac{\pi^2  p y}{6}
.
\end{eqnarray*}
Similarly, we also have 
\begin{eqnarray*}
\sum_{c=1}^\infty \frac{1}{c}\sum_{m\neq 0}e(-mcx)\int_\R \frac{e(mcyt)}{(1+t^2)}dt&=&-\pi\log(|\eta(z)|^2)-\tfrac{\pi^2   y}{6}\end{eqnarray*}
so that 
\begin{eqnarray*}
\cK_{\G,\infty}(z)&=&
\lim_{s\to 1}\left(\frac{(p-1)\sqrt{\pi}\G(s-\tfrac12)\zeta(2s-1)}{(p^{2s}-1)\G(s)\zeta(2s)}y^{1-s}-\frac{3}{\pi(p+1)(s-1)}\right)\\
&&+\frac{6\log(|\eta(z)|^2|\eta(pz)|^{-2p})}{\pi(p^2-1)}
.
\end{eqnarray*}

Next, to compute the limit, write $\zeta(2s-1)=\tfrac{1}{2(s-1)}+\gamma+O(s-1)$ to get that the limit above is given by
\begin{eqnarray*}
\frac{6\g}{(p+1)\pi}+ \lim_{s\to 1}\frac{\tfrac{(p-1)y^{1-s}\sqrt{\pi}\G(s-\tfrac12)}{2(p^{2s}-1)\G(s)\zeta(2s)}-\frac{3}{\pi(p+1)}}{s-1}&=&
\frac{6\g}{(p+1)\pi}+\frac{d}{ds}_{|_{s=1}}\left(\tfrac{(p-1)y^{1-s}\sqrt{\pi}\G(s-\tfrac12)}{2(p^{2s}-1)\G(s)\zeta(2s)}\right)
.
\end{eqnarray*}
Finally, evaluate the derivative at $s=1$ 
\begin{eqnarray*}
\frac{d}{ds}_{|_{s=1}}\left(\tfrac{(p-1)y^{1-s}\sqrt{\pi}\G(s-\tfrac12)}{2(p^{2s}-1)\G(s)\zeta(2s)}\right)
&=&\frac{3}{(p+1)\pi}\left(\log(y)- \frac{2\log(p)p^2}{(p^2-1)}-2\log(2)-2 \frac{\zeta'}{\zeta}(2) \right)
\end{eqnarray*}
to get that 
\begin{eqnarray*}
\cK_{\G,\infty}(z)&=&
\frac{3}{(p+1)\pi}\left(2\g-2 \frac{\zeta'}{\zeta}(2)-\log(\frac{4y|\eta(pz)|^{\frac{4p}{p-1}}}{|\eta(z)|^{\frac{4}{p-1}}})- \frac{2\log(p)p^2}{(p^2-1)} \right).\\
\end{eqnarray*}
The formula for the cusp at $0$ now follows from \eqref{e:KG2G0} and \eqref{eq:EisIs}.
\end{proof}

%

\end{document}